\documentclass[12pt]{article}

\usepackage{amsmath}
\usepackage{amssymb}
\usepackage{amsthm}
\usepackage{amsfonts}
\usepackage{latexsym}
\usepackage{cite}
\usepackage{psfrag}
\usepackage{epsfig}
\usepackage{graphicx}
\usepackage{mathrsfs}
\usepackage{url}
\usepackage{color}
\usepackage{eufrak}
\usepackage{multirow}
\usepackage[T1]{fontenc}

\usepackage[colorlinks=true]{hyperref}
\hypersetup{urlcolor=blue, citecolor=red}

\makeatletter
\@addtoreset{equation}{section}
\makeatother

\newtheorem{theorem}{Theorem}[section]
\newtheorem{lemma}[theorem]{Lemma}
\newtheorem{corollary}[theorem]{Corollary}

\newtheorem{proposition}[theorem]{Proposition}
\theoremstyle{definition}

\newtheorem{notation}[theorem]{Notation}
\newtheorem{remark}[theorem]{Remark}

\newcommand{\Lc}{\mathcal{L}}
\newcommand{\OO}{\mathcal{O}}
\newcommand{\Sc}{\mathcal{S}}
\newcommand{\TT}{\mathcal{T}}

\newcommand{\C}{\mathscr{C}}
\newcommand{\M}{\mathscr{M}}
\newcommand{\N}{\mathscr{N}}
\newcommand{\Os}{\mathscr{O}}
\newcommand{\Ls}{\mathscr{L}}

\newcommand{\PG}{\mathrm{PG}}

\newcommand{\Tr}{\mathrm{T}}
\newcommand{\EnG}{\mathrm{En\Gamma}}

\newcommand{\MM}{\mathbf{M}}
\newcommand{\Pf}{\mathbf{P}}

\newcommand{\F}{\mathbb{F}}
\newcommand{\Pb}{\mathbb{P}}
\newcommand{\Lb}{\mathbb{L}}
\newcommand{\Nb}{\mathbb{N}}
\newcommand{\Ob}{\mathbb{O}}

\newcommand{\Wb}{\mathbb{W}}

\newcommand{\A}{\mathfrak{A}}

\newcommand{\Mk}{\mathfrak{M}}
\newcommand{\Nk}{\mathfrak{N}}
\newcommand{\Pk}{\mathfrak{P}}

\newcommand{\pk}{\mathfrak{p}}
\newcommand{\Rk}{\mathfrak{R}}

\newcommand{\T}{\text}
\newcommand{\db}{\displaybreak[3]}

\begin{document}
\title{Further results on orbits and incidence matrices for the class $\OO_6$ of lines external to the twisted cubic in
$\PG(3,q)$
\date{}
}
\maketitle
\begin{center}
{\sc Alexander A. Davydov}\\
 {\sc\small Kharkevich Institute for Information Transmission Problems}\\
 {\sc\small Russian Academy of Sciences,
Moscow, 127051, Russian Federation}\\
 \emph{E-mail address:} alexander.davydov121@gmail.com\medskip\\
 {\sc Stefano Marcugini and
 Fernanda Pambianco }\\
 {\sc\small Department of  Mathematics  and Computer Science,  Perugia University,}\\
 {\sc\small Perugia, 06123, Italy}\\
 \emph{E-mail address:} \{stefano.marcugini, fernanda.pambianco\}@unipg.it
\end{center}

\begin{abstract}
In the literature, lines of the projective space $\mathrm{PG}(3,q)$ are partitioned into classes, each of which is a union of line orbits under the stabilizer group of the twisted cubic. The least studied class is named $\mathcal{O}_6$. This class contains lines external to the twisted cubic which are not its chords or axes and do not lie in any of its osculating planes. For even and odd $q$, we propose  a new family of orbits of $\mathcal{O}_6$ and investigate in detail their stabilizer groups and the corresponding submatrices of the point-line and plane-line incidence matrices. To obtain these submatrices, we explored the number of solutions of cubic and quartic equations connected with intersections of lines (including the tangents to the twisted cubic), points, and planes in $\mathrm{PG}(3,q)$.
\end{abstract}



\section{Introduction}
In the three-dimensional projective space $\PG(3,q)$ over a Galois field $\F_{q}$ with $q$ elements, the normal rational curve $\mathscr{C}$,  named twisted cubic, has as many as $q+1$ points. Up to a change of the projective frame of $\PG(3,q)$, these points are  $P_t=(t^3,t^2,t,1)$, $t\in \F_{q}$, together with $P_\infty=(1,0,0,0)$. In particular, they form a complete $(q+1)$-arc in $\PG(3,q)$. The twisted cubic has many interesting properties and is connected with distinct combinatorial and applied problems, which led this curve to be widely studied, see for instance \cite{BallicoCos,BonPolvTwCub,BrHirsTwCub,CKS,CLPolvT_Spr,CosHirsStTwCub,GiulVincTwCub,Hirs_PG3q,KorchLanzSon,LawLiPav,LunarPolv,ZanZuan2010}.
A novel application of twisted cubic aimed at the construction of covering codes has been the motivation for the study of certain submatrices of the point-plane incidence matrix of $PG(3,q)$ arising from the action of the stabilizer group $G_q\cong \mathrm{PGL}(2,q)$ of $\mathscr{C}$ in $\PG(3,q)$.
The investigation, based on the known classification of the point and plane orbits of $G_q$ given in \cite{Hirs_PG3q}, was initiated by D. Bartoli and the present authors in 2020 \cite{BDMP-TwCub} and produced optimal multiple covering codes. The results in \cite{BDMP-TwCub} were also an important ingredient to classify the cosets of the $[ q+1, q-3,5]_q 3$ generalized doubly-extended Reed-Solomon code of codimension $4$ by means of their weight distributions \cite{DMP_RSCoset}.

For the study of the plane-line and the point-line incidence matrices, an explicit description of line orbits is necessary. In \cite{Hirs_PG3q}, a partition of the lines in $\PG(3,q)$ into classes is given, each of which is a union of line orbits under $G_q$. Apart from one class denoted by $\OO_6$,  containing lines external to the twisted cubic that are not its chords or axes and do not lie in its osculating planes,
the number and the structure of the orbits forming those unions are simultaneously and independently obtained by distinct methods in
 \cite{DMP_OrbLineMJOM}  (for all $q\ge2$),  \cite{BlokPelSzo} (for all $q\ge23$), and \cite{GulLav}  (for finite fields of characteristic $> 3$); see also the references therein.

The results  of \cite{BlokPelSzo,DMP_OrbLineMJOM,GulLav} are collected in \cite[Sect. 2.2, Tab. 1]{DMP_PointLineInc} where texts from arXiv.org corresponding to \cite{BlokPelSzo,DMP_OrbLineMJOM,GulLav} are used.

Using the representation of the line orbits in \cite{DMP_OrbLineMJOM}, for all $q\ge2$ and apart from the lines in class $\OO_6$, the \emph{point-line} and \emph{plane-line} incidence matrices of $\PG(3,q)$ are obtained in \cite{DMP_PointLineInc,DMP_PlLineIncJG}. 
In these incidence matrices, submatrices correspond to the orbits. In \cite{DMP_PlLineIncJG}, for the
submatrices, the number of  distinct lines lying in distinct planes and, conversely, the number of distinct
 planes through distinct lines are obtained. In \cite{DMP_PointLineInc},
 the number of distinct lines through distinct points and, vice versa, the number of distinct
 points lying on distinct lines are obtained. (By ``distinct planes'' we mean
 ``planes from distinct orbits'', and similarly for points and lines.)

In \cite{GulLav}, apart from the lines in class $\OO_6$, for odd $q\not\equiv0\pmod3$, the numbers of distinct planes through distinct lines (called ``the plane orbit distribution of a line") and the numbers of distinct points lying on distinct lines
(called ``the point orbit distribution of a line") in of $\PG(3,q)$ are obtained.
For finite fields of characteristic $> 3$, the results of \cite{GulLav} on ``the plane orbit distribution of a line" and ``the point orbit distribution of a line'' are in accordance with those from \cite{DMP_PointLineInc,DMP_PlLineIncJG} on the point-line and the plane-line incidence matrices.

In \cite{DMP_OrbLineMJOM} stabilizer groups for the considered orbits are obtained and described in detail whereas in \cite{BlokPelSzo,GulLav} the stabilizer groups are not considered. Also, in \cite[Th. 8.1, Conj. 8.2]{DMP_OrbLineMJOM}, for the all fields $\F_q$, $q\ge5$,  a detailed conjecture on the sizes and the number of line orbits in the class  $\OO_6$ is formulated; for $5\le q\le 37$ and $q=64$ the conjecture has been proved by an exhaustive computer search.

In  $\PG(3,q)$, for $q=2^n$, $n\ge3$,
the $(q+1)$-arc    $\mathcal{A}=\{(1,t,t^{2^h},t^{2^h+1})\,|\,t\in\F_q^+\}$, $\F_q^+=\F_q\cup\{\infty\}$, with $gcd(n,h)=1$ (twisted cubic for $h=1$), has been considered  in a recent paper \cite{CerPavDM},
where it is shown that the orbits of points and of planes
under the projective stabilizer $G_h$ of $\mathcal{A}$ are similar to those under $G_1$ described in \cite{Hirs_PG3q}; moreover, the point-plane incidence matrix with
respect to $G_h$-orbits  mirrors the case h=1 described in  \cite{BDMP-TwCub}. In \cite{CerPavDM}, it is also proved for even $q$,   our conjecture of \cite[Th. 8.1, Conj. 8.2]{DMP_OrbLineMJOM}.

In \cite{DMP_OrbLineO6MJOM}, for all even and odd $q$, a so-called family $\Os_\mu$ of line orbits of the class $\OO_6$ is obtained using a line family, called \emph{$\ell_\mu$-lines}, where a parameter $\mu$ runs over $\F_q^*\setminus\{1\}$, $\F_{q}^*=\F_{q}\setminus\{0\}$. Also, one more  orbit $\Os_\Lc$, based on a line $\Lc$ with another description, is given. The orbits $\Os_\mu$ and $\Os_\Lc$ are based on the analysis of the stabilizer groups of the corresponding lines. These orbits include an essential part of all $\OO_6$ orbits; e.g. they include about one-half and  one-third of all lines of $\OO_6$ for $q$ even and for $q\equiv 0 \pmod3$, respectively.

 In \cite{DMP_IncO6JG}, using the properties of the orbits $\Os_\mu$ and $\Os_\Lc$ from \cite{DMP_OrbLineO6MJOM}, we determine all the plane-line and point-line incidence matrices connected with these orbits.

 In a quite recent paper \cite{KaPatPradOrbits}, for a field $\F_q$ of characteristic $>3$, our conjecture of \cite[Th. 8.1, Conj. 8.2]{DMP_OrbLineMJOM} on the sizes and the number of orbits in the class  $\OO_6$ has been proved. For it,  the open problem of classifying binary quartic forms over $\F_q$
into $G_q$-orbits is solved and used. Also, the Pl\"{u}cker embedding for the Klein quadric is applied. So, in \cite{KaPatPradOrbits}, the methods and approaches are different from our published articles \cite{DMP_PointLineInc,DMP_PlLineIncJG,DMP_OrbLineMJOM,DMP_OrbLineO6MJOM,DMP_IncO6JG} and from this paper. Also, in \cite{KaPatPradOrbits}, unlike this paper, line orbits over the fields $\F_q$ of characteristic 2 are not investigated and the incidence matrices are not considered.

In this paper, the first version of which can be found in
\cite{DMP_arXivLrho2023}, we continue and develop approaches of \cite{DMP_OrbLineO6MJOM,DMP_IncO6JG}. We propose a new family of lines $\Ls_\rho$, where $\rho$ is a parameter running over $\F_q^*$. This family is a generalization of the line $\Lc$ from \cite{DMP_OrbLineO6MJOM}. For even and odd $q\not\equiv 0 \pmod3$ the lines $\Ls_\rho$ belong to the class $\OO_6$. The detailed investigation of the stabilizer groups of the lines $\Ls_\rho$ for all $q$ and $\rho$ allows  us to calculate the sizes of the orbits under $G_q$, containing the lines $\Ls_\rho$. Also, the parameters of the point-line and plane-line incidence submatrices are being considered with the help of the research of the number of solutions of cubic and quartic equations connected with intersections of the lines $\Ls_\rho$ and distinct planes, points, and the tangents to the twisted cubic. Formulas for the numbers
of the solutions are being described. It is shown when the lines $\Ls_\rho$ generate new orbits in comparison with the orbits $\Os_\mu$ \cite{DMP_OrbLineO6MJOM}.

The paper is organized as follows. Section \ref{sec2:prelim} is background and preliminaries. In Section \ref{sec3:Lrho}, the new family of lines $\Ls_\rho$ from the class $\OO_6$ is described. Some useful relations are given in Section \ref{sec:useful}. In Section \ref{sec5:intersec}, we investigate intersections of $\Ls_{\rho}$-lines and tangents. In Section \ref{sec6}, the stabilizers of $\Ls_\rho$-lines and the sizes of the orbits are obtained. Cubic equations and incidence matrices for even and odd $q$ are given in Sections \ref{sec7:cubEqEven} and \ref{sec8:cubEqIncOdd q}, respectively. Orbits, generated by $\Ls_{\rho}$-lines, for even and odd $q$, are considered in Section \ref{sec8:orbits even} and \ref{sec10:orbits odd q}, respectively.\\
An extended abstract of this paper appeared in WCC 2024 Proceedings \cite{WCC2024}.

\section{Preliminaries}\label{sec2:prelim}
\subsection{Twisted cubic in $\boldsymbol{\PG(3,q)}$}\label{subsec21:twcub}
Here we cite some results from \cite{Hirs_PG3q,Hirs_PGFF} useful in this paper.

Let $\boldsymbol{\pi}(c_0,c_1,c_2,c_3)$ be the plane of $\PG(3,q)$ with equation
$c_0x_0+c_1x_1+c_2x_2+c_3x_3=0$, $c_i\in\F_q$.
Let $\Pf(x_0,x_1,x_2,x_3)$ be a point of $\PG(3,q)$  with homogeneous coordinates $x_i\in\F_{q}$.
Let  $P(t)$ be a point of $\PG(3,q)$ with
\begin{equation}\label{eq21:Pt}
  t\in\F_q^+,~ P(t)=\Pf(t^3,t^2,t,1) \T{ if }t\in\F_q,~P(\infty)=\Pf(1,0,0,0).
\end{equation}

Let $\C\subset\PG(3,q)$ be the \emph{twisted cubic} consisting of $q+1$ points no four of which are coplanar.
We consider $\C$ in the canonical form
\begin{equation}\label{eq21:cubic}
\C=\{P(t)\,|\,t\in\F_q^+\}.
\end{equation}

A \emph{chord} of $\C$ through the points $P(t_1)$ and $P(t_2)$ is a line joining either a pair of real points of
$\C$, possibly coincident, or a pair of complex conjugate points. Its coordinate vector is
$ L_{\T{ch}}=(a_2^2, a_1 a_2, a_1^2-a_2, a_2, -a_1, 1)$,
where $a_1 =t_ 1 + t_ 2$, $a_2 =t_1 t_ 2$. If $x^2- a_1x + a_2$ has 2, 1, or 0  roots in $\F_q$ then we have, respectively, a real chord, a tangent, or an imaginary chord.
The \emph{tangent} $\TT_t$ to $\C$ at the point $P(t)$ has coordinate vector
\begin{align}\label{eq21:tang}
& L^{\T{tang}}_t=(t^4, 2t^3, 3t^2, t^2, -2t, 1),~t\in\F_q;~L^{\T{tang}}_\infty=(1, 0, 0, 0, 0, 0).
\end{align}

The \emph{osculating plane} $\pi_\T{osc}(t)$ in the  point $P(t)\in\C$ has the form
\begin{equation}\label{eq21:osc_plane}
\pi_\T{osc}(t)=\boldsymbol{\pi}(1,-3t,3t^2,-t^3)\T{ if }t\in\F_q; ~\pi_\T{osc}(\infty)=\boldsymbol{\pi}(0,0,0,1).
\end{equation}
 The $q+1$ osculating planes form the osculating developable $\Gamma$ to $\C$, that is a \emph{pencil of planes} for $q\equiv0\pmod3$ or a \emph{cubic developable} for $q\not\equiv0\pmod3$.

The null polarity $\A$ \cite[Theorem~21.1.2]{Hirs_PG3q}, \cite[Sections 2.1.5, 5.3]{Hirs_PGFF} is given by
\begin{equation}\label{eq21:null_pol}
\Pf(x_0,x_1,x_2,x_3)\A=\boldsymbol{\pi}(x_3,-3x_2,3x_1,-x_0),~q\not\equiv0\pmod3.
\end{equation}

For $q\not\equiv0\pmod3$, dual to the chords of $\C$ are the axes of $\Gamma$. An \emph{axis} of $\Gamma$ is a line of $\PG(3,q)$ which is the intersection of a pair of real planes or complex conjugate planes of~$\Gamma$. In the last  case it is an \emph{imaginary axis}. If the real planes are distinct it is a \emph{real axis}; if they coincide with each other, it is a tangent to $\C$. An axis has coordinate vector $L_{\T{ax}}=(\beta_2^2, \beta_1 \beta_2, 3\beta_2, (\beta_1^2-\beta_2)/3, -\beta_1,1)$, $\beta_i\in\F_q^+$, $q\not\equiv0\pmod3$.

\begin{notation}\label{notation21}
Throughout the paper, we consider $q\equiv\xi\pmod3$ with $\xi\in\{-1,0,1\}$. Many values depend on $\xi$ or make sense only for specific $\xi$.
If it is not clear by the context, we note this by remarks.
The following notation is used.
\begin{align*}
  &G_q && \T{the group of projectivities in } \PG(3,q) \T{ fixing }\C;\db  \\
&\#S&&\T{the cardinality of a set }S;\db\\
&\overline{AB}&&\T{the line through the points $A$ and }B;\db\\
&\triangleq&&\T{the sign ``equality by definition''}.\db\\
&&&\textbf{Types $\pi$ of planes:}\db\\
&\Gamma\T{-plane}  &&\T{an osculating plane of }\Gamma;\db \\
&d_\C\T{-plane}&&\T{a plane containing \emph{exactly} $d_\C$ distinct points of }\C,~d_\C\in\{0,2,3\};\db \\
&\overline{1_\C}\T{-plane}&&\T{a plane not in $\Gamma$ containing \emph{exactly} 1 point of }\C;\db \\
&\Pk&&\T{the list of possible types $\pi$ of planes},~\Pk\triangleq\{\Gamma,2_\C,3_\C,\overline{1_\C},0_\C\};\db\\
&\pi\T{-plane}&&\T{a plane of the type }\pi\in\Pk.\db\\
&&&\textbf{Types $\pk$ of points, $\xi\ne0$:}\db\\
&\C\T{-point}&&\T{a point  of }\C;\db\\
&\mu_\Gamma\T{-point}&&\T{a point  off $\C$ lying on \emph{exactly} $\mu$ distinct osculating planes},\db\\
&&&\mu_\Gamma\in\{0_\Gamma,1_\Gamma,3_\Gamma\}\db\\
&\Tr\T{-point}&&\T{a point  off $\C$  on a tangent to $\C$};\db\\
&\Mk&&\T{the list of possible types $\pk$ of points},\Mk\triangleq\{\C,0_\Gamma,1_\Gamma,3_\Gamma,\Tr\};\db\\
&\pk\T{-point}&&\T{a point of the type }\pk\in\Mk.\db\\
&&&\textbf{Orbits under $G_q$:}\db\\
&\N_\pi&&\T{the orbit of $\pi$-planes under }G_q,~\pi\in\Pk;\db\\
&\M_\pk&&\T{the orbit of $\pk$-points under }G_q,~\pk\in\Mk;\db\\
&\EnG\T{-line}&&\T{a line, external to the cubic $\C$, not in a $\Gamma$-plane, that is neither}\db\\
&&&\T{a chord nor an axis, see \cite[Lemma 21.1.4]{Hirs_PG3q} and its context;}\db \\
&\OO_6=\OO_{\EnG}&&\T{the union (class) of all orbits of $\EnG$-lines}.
\end{align*}
\end{notation}

\begin{theorem}\label{th21:Hirs}
\emph{\cite[Chapter 21]{Hirs_PG3q}} The following properties of the twisted cubic $\C$ hold:
\begin{description}
  \item[(i)] The group $G_q$ acts triply transitively on $\C$;  $G_q\cong PGL(2,q)$ for $q\ge5$.
    A matrix $\MM$ corresponding to a projectivity of $G_q$ has the general form
  \begin{align}\label{eq21:M}
& \mathbf{M}=\left[
 \begin{array}{cccc}
 a^3&a^2c&ac^2&c^3\\
 3a^2b&a^2d+2abc&bc^2+2acd&3c^2d\\
 3ab^2&b^2c+2abd&ad^2+2bcd&3cd^2\\
 b^3&b^2d&bd^2&d^3
 \end{array}
  \right],a,b,c,d\in\F_q, ad-bc\ne0.
\end{align}

  \item[(ii)] Under $G_q$, $q\ge5$, there are the following five orbits $\N_\pi$ of planes:
\begin{align*}
   &\N_1=\N_\Gamma=\{\Gamma\T{-planes}\},\,\#\N_\Gamma=q+1;~\N_{2}=\N_{2_\C}=\{2_\C\T{-planes}\},\,\#\N_{2_\C}=q^2+q;\db \\
 &\N_{3}=\N_{3_\C}=\{3_\C\T{-planes}\},~  \#\N_{3_\C}=(q^3-q)/6;~\N_{4}=\N_{\overline{1_\C}}=\{\overline{1_\C}\T{-planes}\},\db\\
 &\#\N_{\overline{1_\C}}=(q^3-q)/2;~\N_{5}=\N_{0_\C}=\{0_\C\T{-planes}\},~\#\N_{0_\C}=(q^3-q)/3.
 \end{align*}

  \item[(iii)] For $q\not\equiv0\pmod 3$, there are the following five orbits $\M_j$ of points:
  \begin{align}
&\M_1=\M_\C=\{\C\T{-points}\},~\M_2=\M_\Tr=\{\Tr\T{-points}\},\db\notag\\
&\M_3=\M_{3_\Gamma}=\{3_\Gamma\T{-points}\},~\M_4=\M_{1_\Gamma}=\{1_\Gamma\T{-points}\},~\M_5=\M_{0_\Gamma}=\{0_\Gamma\T{-points}\}.\db\notag\\
&  \M_j\A=\N_j,~\#\M_j=\#\N_j,~j=1,\ldots,5; \label{eq21:sizeMj=Nj}\db\\
&\M_\C\A=\N_\Gamma,~\M_\Tr\A=\N_{2_\C}, ~
\M_{3_\Gamma}\A=\N_{3_\C},~\M_{1_\Gamma}\A=\N_{\overline{1_\C}},~
\M_{0_\Gamma}\A=\N_{0_\C}.\label{eq21:pi(pk)}
\end{align}

 \item[(iv)]  The lines of $\PG(3,q)$ can be partitioned into classes called $\OO_i$ and $\OO'_i=\OO_i\A$, each of which is a union of orbits under $G_q$. The full list of the classes can be found in \emph{\cite[Lemma 21.1.4]{Hirs_PG3q}}. In particular, for all $q$, there is the class $\OO_6=\OO_{\EnG}=\{\EnG\T{-lines}\}$, $\#\OO_6=\#\OO_{\EnG}=(q^2-q)(q^2-1)$. If  $q\not\equiv0\pmod3,$ we have $\OO_6=\OO'_6=\OO_6\A$.
\end{description}
\end{theorem}

\subsection{An $\boldsymbol{\EnG}$-line $\boldsymbol{\Lc}$ and its orbit $\boldsymbol{\Os_\Lc}$, $\boldsymbol{q\not\equiv0\pmod3}$}\label{subsec22:OL}
Here we present some results from \cite{DMP_OrbLineO6MJOM} useful in this paper.

Let  $Q_\beta=\Pf(1,0,\beta,1), ~\beta\in\F_q^+$, be a point of $\PG(3,q)$. We consider the line
\begin{equation}\label{eq22:ell0infdef}
 \Lc = \overline{Q_0Q_\infty} =\overline{\Pf(1,0,0,1)\Pf(0,0,1,0)}= \{\Pf(0,0,1,0), \Pf(1,0,\beta,1)\,|\,\beta\in\F_q\}.
\end{equation}
Let $\Os_\Lc$ be the orbit of the line $\Lc$ under $G_q$.

\begin{theorem}\label{th22:orbitell0inf}
\emph{\cite[Section 3]{DMP_OrbLineO6MJOM}}
  Let $q\equiv\xi\pmod3$, $\xi\ne0$. We have $\Os_\Lc\subset\OO_6=\OO_{\EnG}$, i.e. the lines of $\Os_\Lc$ are $\EnG$-lines.
 The orbit $\Os_\Lc$ has size
\begin{align*}
 \#\Os_\Lc=\left\{\begin{array}{@{\,}lccl}
                          (q^3-q)/3 & \T{if} & \xi=1, &q\T{ is even or $2$ is a non-cube in }\F_q;\\
                          (q^3-q)/12&\T{if} &\xi=1,  &q \T{ is odd and $2$ is a cube in }\F_q; \\
                          q^3-q& \T{if}&\xi=-1,  &q\T{ is even}; \\
                          (q^3-q)/2& \T{if}&\xi=-1,  &q\T{ is odd}.
                        \end{array}
 \right.
\end{align*}
\end{theorem}

\subsection{A family of $\boldsymbol{\EnG}$-lines $\boldsymbol{\ell_\mu}$ and their orbits $\boldsymbol{\Os_\mu}$}\label{subsec23:lmu}
Here we cite some results from \cite{DMP_OrbLineO6MJOM,DMP_IncO6JG} useful in this paper.

 Let $ \mu\in \F_q^*\setminus\{1\}$ if $q$ is even or $q\equiv0\pmod3$;  $\mu\in
                 \F_q^*\setminus\{1,1/9\}$ if $q$ is odd, $q \not\equiv0\pmod3$. Let
  $ R_{\mu,\gamma}=\Pf(\gamma,\mu,\gamma,1)$, $\gamma\in\F_q^+$, be a point of $\PG(3,q)$.
We consider the line
\begin{equation*}
\ell_{\mu}= \overline{R_{\mu,0}R_{\mu,\infty}} = \overline{\Pf(0,\mu,0,1)\Pf(1,0,1,0)}= \{\Pf(\gamma,\mu,\gamma,1)|\gamma\in\F_q^+,~\mu\T{ is fixed}\}.
\end{equation*}
Let $\Os_{\mu}$ be the orbit of the line $\ell_\mu$ under $G_q$.

\begin{theorem}\label{th23:orbitellmu}
\emph{\cite[Sections 4--7]{DMP_OrbLineO6MJOM}, \cite[Sections 5--7]{DMP_IncO6JG}} Let $q\equiv\xi\pmod3$.
\begin{description}
  \item[(i)]  For all $q\ge5$, we have $\Os_\mu\subset\OO_6=\OO_{\EnG}$, i.e. the lines of $\Os_\mu$ are $\EnG$-lines.
The orbit $\Os_\mu$ has size
\begin{align*}
 \#\Os_\mu=\left\{\begin{array}{@{\,}ll}
                          (q^3-q)/2& \T{if } q\T{ is even or }\mu \T{ is a non-square in }\F_q;\\
                          (q^3-q)/4& \T{if }\mu \T{ is a square in }\F_q\T{ and }\xi=0;\\
                          (q^3-q)/4&\T{if }q\T{ is odd},  \mu \T{ is a square in }\F_q,~\xi\ne0,\Upsilon_{q,\mu}\T{ does not hold};\\
                          (q^3-q)/12&\T{if }q\T{ is odd},~\xi\ne0,  \Upsilon_{q,\mu}\T{ holds} ;
                        \end{array}
 \right.
\end{align*}
where the condition $\Upsilon_{q,\mu}$ has the form
\begin{equation}\label{eq23:Upsilon}
  \Upsilon_{q,\mu} \;:~\mu=-1/3,~q \equiv 1 \pmod {12}, ~-1/3\T{ is a fourth power}.
\end{equation}
  \item[(ii)] Let $q$ be odd, $\xi\ne0$, $\mu\in
                 \F_q^*\setminus\{1,1/9\}$. Then every line of $\Os_{\mu}$ contains $\mathfrak{n}_{q}(\mu)$ $\Tr$-points;
 \begin{align}\label{eq23:solution tang}
 &\mathfrak{n}_{q}(\mu)=\#\left\{t\in\F_q\,|\;t=\pm\sqrt{\frac{1}{2}\left(3\mu-1\pm\sqrt{(\mu-1)(9\mu-1}\right)}\right\}\in\{0,2,4\}.\
 \end{align}

  \item[(iii)] Let $q$ be even. Then every line of the orbit $\Os_{\mu}$ contains $\mathfrak{n}_{q}(\mu)=2$ $\Tr$-points.
\end{description}
\end{theorem}

\section{A family of $\boldsymbol{\EnG}$-lines $\boldsymbol{\Ls_\rho, \rho \neq 0}$}\label{sec3:Lrho}
We introduce an $\EnG$-line $\Ls_\rho$.
Let $K_{\rho,\gamma}$ be the points of $\PG(3,q)$ such that
\begin{equation}\label{eq3:K}
 K_{\rho,\gamma}=\Pf(\rho,0,\gamma,1),~\gamma\in\F_q^+,~\rho\in\F_q^*;~K_{\rho,0}=\Pf(\rho,0,0,1),~K_{\rho,\infty}=\Pf(0,0,1,0).
\end{equation}
We consider the line $\Ls_{\rho}$ through $K_{\rho,0}$ and $K_{\rho,\infty}$.
\begin{align}\label{eq3:Lrho}
&\Ls_{\rho}= \overline{\Pf(\rho,0,0,1)\Pf(0,0,1,0)}= \{\Pf(\rho,0,\gamma,1)|\gamma\in\F_q^+,~\rho\T{ is fixed},~\rho\in\F_q^*\}.
\end{align}
By \eqref{eq3:Lrho}, the coordinate vector $L_{\rho}$ of $\Ls_{\rho}$  is
\begin{equation}\label{eq3:cordVecElrho}
  L_{\rho} = (0,\rho,0,0,0,-1).
\end{equation}

By \eqref{eq21:tang}, \eqref{eq3:cordVecElrho}, if $\rho=0$ the line $\Ls_0$ is the tangent $\TT_0$.
Note also that the equations of $\TT_0$ are $x_0=0,x_1=0$ \cite[Lemma 5.2]{DMP_OrbLineMJOM}. This explains why we consider $\rho\in \F_q^*$.

By \eqref{eq22:ell0infdef}, \eqref{eq3:Lrho}, the line $\Lc$ of \cite{DMP_OrbLineO6MJOM} is the line $\Ls_1$.
\begin{lemma}\label{lem3:pi0inf}
Let $q\not\equiv0\pmod3$.  Then we have
\begin{align}\label{eq3:K_in notin}
&  K_{\rho,\gamma}\notin\pi_\T{osc}(t),~ t=0,\infty,~\gamma\in\F_q;\db\\
& K_{\rho,\infty}\in\pi_\T{osc}(t),~ t=0,\infty;~ K_{\rho,\infty}\notin\pi_\T{osc}(t),~t\in\F_q^*;\db\notag\\
&K_{\rho,0}\in\pi_\T{osc}(t),~t\in\F_q^*,\T{ if and only if }\rho=t^3.\notag
\end{align}
This implies the following:
$K_{\rho,\infty}$ belongs to exactly two osculating planes, i.e. $K_{\rho,\infty}$ is a $\Tr$-point.
If $\rho$ is a non-cube in $\F_q$ then $K_{\rho,0}$ is a $0_\Gamma$-point.
If $\rho$ is a cube in $\F_q$ then $K_{\rho,0}$ is a $1_\Gamma$-point, if $q\equiv-1\pmod3$, or $3_\Gamma$-point, if $q\equiv1\pmod3$.
\end{lemma}

\begin{proof}
 By \eqref{eq21:osc_plane},\eqref{eq3:K}, we have \eqref{eq3:K_in notin}. The rest of the assertions
follows from \cite[Section 1.5(iv)(v)]{Hirs_PGFF}.
\end{proof}

\begin{lemma}\label{lem3:l0infInEnG}
\begin{description}
  \item[(i)]
For $q\not\equiv0\pmod3$, the line  $\Ls_{\rho}$ is an $\EnG$-line.
  \item[(ii)]
For $q\equiv0\pmod3$, the line  $\Ls_{\rho}$ is not an $\EnG$-line.
\end{description}
\end{lemma}

\begin{proof}
\begin{description}
  \item[(i)]
By Lemma \ref{lem3:pi0inf}, for each osculating plane there is a point of $\Ls_{\rho}$ not belonging to it.
Also, comparing the coordinate vectors \eqref{eq3:cordVecElrho} and $L_{\T{ch}}$  (resp. $L_{\T{ax}}$) from Section~\ref{subsec21:twcub}, we obtain $a_1=a_2=0,~ a_1a_2=\rho$ (resp. $\beta_1=\beta_2=0, ~\beta_1\beta_2=\rho$), contradiction. So, $\Ls_{\rho}$ is not a chord (resp. axis) of $\C$.
  \item[(ii)]
  By \eqref{eq21:osc_plane}, \eqref{eq3:Lrho}, for $q\equiv0\pmod3$, the line  $\Ls_{\rho}$ lies in the osculating plane $\pi_\T{osc}(\sqrt[3]{\rho})$.
  Also, if $q\equiv0\pmod3$, then every element of $\F_q$ is a cube. \qedhere
\end{description}
\end{proof}

\emph{From now on, we consider the lines $\Ls_{\rho}$ for $q\not\equiv0\pmod3$.
}
\begin{lemma}\label{lem3:LsU=Ls}
 We have $\Ls_\rho=\Ls_\rho\A$.
 \end{lemma}

 \begin{proof}
 We have $\Ls_\rho\A=K_{\rho,\infty}\A\cap K_{\rho,0}\A=\boldsymbol{\pi}(0,-3,0,0)\cap\boldsymbol{\pi}(1,0,0,-\rho)$ that implies
 $\{K_{\rho,\infty}$, $K_{\rho,0}\}\subset\Ls_\rho\A$.
 \end{proof}

\section{Useful relations}\label{sec:useful}

We fix  a primitive element $\alpha$ of the field $\F_q$. The discrete logarithm $\log$ of $\beta \in \F_q^*$ is the integer $b \in [0, \dots, q-1]$ such that $\alpha^b = \beta$.
Let $\Rk_m$, $m=0,1,2$, be a class of the values of $\rho$ such that
    \begin{align}\label{eq4:Rmdefin}
 &     \Rk_m\triangleq\{\rho\in \F^*_q\,|\,\log \rho\equiv m\pmod3\}.
    \end{align}

\begin{lemma}\label{lem82:logofcubes}
Let $q\equiv1\pmod3$. Then $\beta \in \F_q^*$ is a cube if and only if $\log\beta\equiv0\pmod3$.
\end{lemma}

\begin{proof}
If $q\equiv1\pmod3$, then $q-1\equiv0\pmod3$. By hypothesis $\beta=\gamma^3$, $\gamma \in \F_q^*$. Then  $\log\beta=3\log\gamma-k(q-1),~ k \ge 0$, so $\log\beta\equiv0\pmod3$. The converse is obvious.
\end{proof}

\begin{lemma}\label{lem4:Tpoint=2Gammapoint}
  For $q\not\equiv0\pmod3$,  every $\Tr$-point lies in exactly two $\Gamma$-planes.
\end{lemma}

\begin{proof}
  The assertion follows from \cite[Table 1]{BDMP-TwCub}.
\end{proof}

 \begin{notation}\label{notat4:}
 The following notation is used:
 \begin{align*}
 &\Ob_\rho&&\T{the orbit under $G_q$ generated by the line }\Ls_\rho;\db\\
&\Pi_{\pi}&&\T{the number of $\pi$-planes through a line from the orbit }\Ob_\rho,~\pi\in\Pk;\db\\
&\Lambda_{\pi}&&\T{the number of lines from the orbit $\Ob_\rho$ in a $\pi$-plane},~\pi\in\Pk;\db\\
&\Pb_{\pk}&&\T{the  number of $\pk$-points on a line from the orbit }\Ob_\rho,~\pk\in\Mk;\db\\
&\Lb_{\pk}&&\T{the number of lines from the orbit $\Ob_\rho$ through a $\pk$-point},\pk\in\Mk.
\end{align*}
Below, for the notations $\Pi_{\pi},\Lambda_{\pi},\Pb_{\pk},\Lb_{\pk}$ the value of $\rho$ will be clear by the context.
\end{notation}

\begin{lemma}\label{lem4:orb the same}
\emph{\cite[Lemma 1(i)(iii)]{DMP_PointLineInc}, \cite[Lemma 4.1(i)]{DMP_PlLineIncJG}}
\begin{description}
  \item[(i)]
  The number  of lines from an orbit $\Ob_\rho$ in a plane of an orbit $\N_\pi$ is the same for all planes of~$\N_\pi$;
conversely, the number of planes from an orbit $\N_\pi$ through a line of an orbit $\Ob_\rho$ is the same for all lines of $\Ob_\rho$. Here $\pi\in\Pk$.

\item[(ii)]
  The number of lines from an orbit $\Ob_\rho$ through a point of an orbit $\M_\pk$ is the same for all points of~$\M_\pk$.
And, vice versa, the number of points from an orbit $\M_\pk$ on a line of an orbit $\Ob_\rho$ is the same for all lines of $\Ob_\rho$. Here $\pk\in\Mk$.
\end{description}
\end{lemma}

\begin{proposition}\emph{\cite[Equation (4.7)]{DMP_PointLineInc}, \cite[Equation (4.3)]{DMP_PlLineIncJG}, \cite[Lemma 3.2, Proposition~3.7]{DMP_IncO6JG}}
For the orbit $\Ob_\rho$, generated by a line $\Ls_\rho$, we have
 \begin{align}
&\Lb_{\pk}=\frac{\Pb_{\pk}\cdot\#\Ob_\rho}{\#\M_\pk},\,\pk\in\Mk;~\Lambda_{\pi}=\frac{\Pi_{\pi}\cdot\#\Ob_\rho}{\#\N_\pi},\,\pi\in\Pk;\label{eq4:obtainLamb1}\db\\
&\Pb_{3_\Gamma}=\frac{q+1-\Pb_{1_\Gamma}-2\Pb_{\Tr}}{3},~\Pb_{0_\Gamma}=\Pb_{\Tr}+2\Pb_{3_\Gamma}.\label{eq4:EnGorbit3}
\end{align}
\end{proposition}

\begin{proposition}\label{prop4:Pb=Pirho}
For the orbit $\Ob_\rho$, generated by a line $\Ls_\rho$, we have
 \begin{align}
&\Pi_{\Gamma}=\Lambda_{\Gamma}=\Pb_{\C}=\Lb_{\C}=0;\label{eq4:=0}\db\\
&\Pb_{\Tr}=\Pi_{2_\C},~\Pb_{0_\Gamma}=\Pi_{0_\C},~\Pb_{1_\Gamma}=\Pi_{\overline{1_\C}},
 ~\Pb_{3_\Gamma}=\Pi_{3_\C};\label{eq4:Pb=Pi}\db\\
&\Lb_{\Tr}=\Lambda_{2_\C},~\Lb_{0_\Gamma}=\Lambda_{0_\C},~\Lb_{1_\Gamma}=\Lambda_{\overline{1_\C}},
 ~\Lb_{3_\Gamma}=\Lambda_{3_\C}.\label{eq4:Lb=Lambda}
\end{align}
\end{proposition}

\begin{proof}
The relation \eqref{eq4:=0} follows from the definition of $\EnG$-lines.

For \eqref{eq4:Pb=Pi} we assume that
 $\ell$ is a line, $\Pi_\pi(\ell\A)$ is the number of $\pi$-planes through $\ell\A$, $\Pb_\pk(\ell)$ is the number of $\pk$-points on $\ell$, where $\pi\in\Pk$, $\pk\in\Mk$. From $\N_{2_\C}=\M_{\Tr}\A$, $\N_{3_\C}=\M_{3_\Gamma}\A$, $\N_{\overline{1_\C}}=\M_{1_\Gamma}\A$, $\N_{0_\C}=\M_{0_\Gamma}\A$, see \eqref{eq21:pi(pk)}, we obtain
 $\Pi_{2_\C}(\ell\A)=\Pb_{\Tr}(\ell),~  \Pi_{3_\C}(\ell\A)=\Pb_{3_\Gamma}(\ell),~\Pi_{\overline{1_\C}}(\ell\A)=\Pb_{1_\Gamma}(\ell),~\Pi_{0_\C}(\ell\A)=\Pb_{0_\Gamma}(\ell)$.
Finally, we use Lemmas \ref{lem3:LsU=Ls} and \ref{lem4:orb the same}.

The relation \eqref{eq4:Lb=Lambda} follows from \eqref{eq21:sizeMj=Nj}, \eqref{eq4:obtainLamb1}, \eqref{eq4:Pb=Pi}.
\end{proof}

\begin{theorem}\label{th4:the same}
  For orbits $\Ob_\rho$, generated by lines $\Ls_\rho$, the plane-line incidence matrix contains, according to \eqref{eq4:=0}--\eqref{eq4:Lb=Lambda},  the same values of the point-line incidence matrix, but in this
case they refer to $\Pi_\pi, \Lambda_\pi$ instead of $\Pb_\pk, \Lb_\pk$.
  \end{theorem}

  \begin{proof}
    The assertion directly follows from Proposition \ref{prop4:Pb=Pirho}.
  \end{proof}


\section{Intersections of $\boldsymbol{\Ls_{\rho}}$-lines and tangents}\label{sec5:intersec}

By \eqref{eq21:tang}, \eqref{eq3:cordVecElrho}, the mutual invariant \cite[Section 15.2]{Hirs_PG3q} of $\Ls_{\rho}$ and the tangent $\TT_t$ to the cubic $\C$ at the point $P(t)$ is
\begin{align*}
  &\varpi(\Ls_{\rho},\TT_t)=-2\rho t-t^4,~t\in\F_q,~\rho\in\F_q^*;~\varpi(\Ls_\rho,\TT_\infty)=-1\ne0.
\end{align*}
 The lines  $\Ls_{\rho}$ and $\TT_t$ intersect if and only if $\varpi(\Ls_{\rho},\TT_t)=0$. Thus, $\Ls_{\rho}$ and $\TT_\infty$ do not intersect; we may consider only intersections of $\Ls_{\rho}$ and $\TT_t$ with $t\in\F_q$.
 The equation $\varpi(\Ls_{\rho},\TT_t)=0$ has the form $t^4+ 2\rho t=0$; we denote the number of its solutions:
 \begin{equation}\label{eq5:nq(rho)}
   \mathfrak{n}_q(\rho)\triangleq\#\{t\,|t^4+2\rho t=0,~t\in\F_q,~\rho\in\F_q^*,~q\not\equiv0\pmod3\}.
 \end{equation}
 By above, we have Lemma \ref{lem5:roots}.
 \begin{lemma}\label{lem5:roots}
 \begin{description}
  \item[(i)]
 The number $\mathfrak{n}_q(\rho)$ of the solutions of the equation $t^4+ 2\rho t=0$  is equal to the number of $\Tr$-points on the $\Ls_{\rho}$-line.
  \item[(ii)]
The roots of the equation $\varpi(\Ls_{\rho},\TT_t)=0$ are as follows:
 \begin{align}\label{eq5:roots}
  &\bullet~ t=0, \T{ if $q$ is even}\db\\
  &\phantom{\bullet~ t=0, ~} \T{or $q$ is odd, $q\not\equiv0\pmod3$, and $-2\rho$ is a non-cube in }\F_q;\notag\db\\
  &\bullet~t=0,~t=\sqrt[3]{-2\rho},\T{ if $q$ is odd, $q\not\equiv0\pmod3$, and $-2\rho$ is a cube in }\F_q.\label{eq5:roots2}
 \end{align}
  \end{description}
 \end{lemma}
 For $\beta\in\F_q$, we define the quadratic character $\eta$ of $\F_q^*$ extended to $\F_q$ as follows:
\begin{equation}\label{eq5:quadr-char}
  \eta(\beta)=\left\{\begin{array}{cl}
                   1 & \T{ if }  \beta\T{ is a square of an element in }\F_q^* \\
                   0 & \T{ if }  \beta=0 \\
                   -1 & \T{ otherwise }
                 \end{array}
  \right..
\end{equation}

 \begin{lemma}\label{lem5:nqmu}
   Let $q$ be odd, $q\not\equiv0\pmod3$. Let $\mathfrak{n}_q(\mu)$ be as in \eqref{eq23:solution tang}. Let $\Sc\triangleq(\mu-1)\cdot(9\mu-1)$, $A_+\triangleq\frac{1}{2}\left(3\mu-1+\sqrt{\Sc}\right)$, $A_-\triangleq\frac{1}{2}\left(3\mu-1-\sqrt{\Sc}\right)$. Then $\mathfrak{n}_q(\mu)\ne1$ and
 \begin{equation}\label{eq5:nqmu}
  \mathfrak{n}_q(\mu)=\left\{\begin{array}{ccl}
                               0 & \T{if} & \eta(\Sc)=-1\T{ or } \eta(S)=1,~ \eta(A_+)=\eta(A_-)=-1;\\
                               2 & \T{if} & \eta(\Sc)=1,~\eta(A_+)=1,~\eta(A_-)= -1\T{ or } \\
                               && \eta(\Sc)=1,~\eta(A_+)=-1,~\eta(A_-)= 1;\\
                               4 & \T{if} & \eta(\Sc)=\eta(A_+)=\eta(A_-)=1.
                             \end{array}
  \right.
 \end{equation}
 \end{lemma}

 \begin{proof} As $\mu\in\F_q^*\setminus\{1,1/9\}$, we have $\Sc\ne0$. Moreover,
    $\Sc=(3\mu-1)^2-4\mu$. If $\eta(\Sc)=1$ and $A_-= 0$ then $(3\mu-1)^2=(3\mu-1)^2-4\mu$, contradiction. So, $A_-\ne0$. Similarly, $A_+\ne0$. Therefore, $\mathfrak{n}_q(\mu)\ne1$. The assertion \eqref{eq5:nqmu} directly follows from \eqref{eq23:solution tang}.
 \end{proof}

%

\begin{theorem}\label{th5:Lrho&tang}
\begin{description}
  \item[(i)] Let $\mathfrak{n}_q(\mu)$ be as in Theorem \emph{\ref{th23:orbitellmu}(iv)(v)} and Lemma \emph{\ref{lem5:nqmu}}. Let $\mathfrak{n}_q(\rho)$ be as in \eqref{eq5:nq(rho)}. If $\mathfrak{n}_q(\rho)\ne\mathfrak{n}_q(\mu)$ then the orbits $\Ob_\rho$ and $\Os_\mu$ are distinct. In general, the opposite is not true.
  \item[(ii)] For all $q$ and $\rho$, the point $K_{\rho,\infty}$ lies on the tangents $\TT_0$. No other points of $\Ls_{\rho}$ belong to  $\TT_0$. This case corresponds to the root $t=0$ of the equation $\varpi(\Ls_{\rho},\TT_t)=0$.

  \item[(iii)] Let $q$ be even or $q$ be odd, $q\not\equiv0\pmod3$, and $-2\rho$ be a non-cube in $\F_q$. Then $\mathfrak{n}_q(\rho)=1$ and $K_{\rho,\infty}$ is the unique point of $\Ls_{\rho}$ lying on a tangent to $\C$. Moreover, every orbit $\Ob_\rho$ is different from any orbit $\Os_\mu$.

  \item[(iv)]Let $q$ be odd, $q\not\equiv0\pmod3$. Let also $-2\rho$ be a cube in $\F_q$.
\begin{description}
  \item[(a)]  Let  $q\equiv-1\pmod3$. The equation $\varpi(\Ls_{\rho},\TT_t)=0$ has two roots ($t=0$ and the unique value of $t=\sqrt[3]{-2\rho}$), i.e. $\mathfrak{n}_q(\rho)=2$.

  \item[(b)]  Let  $q\equiv1\pmod3$. The equation $\varpi(\Ls_{\rho},\TT_t)=0$ has $4$ roots ($t=0$ and three distinct values  of $t=\sqrt[3]{-2\rho}$), i.e. $\mathfrak{n}_q(\rho)=4$.
\end{description}
\end{description}
 \end{theorem}

\begin{proof}
\begin{description}
  \item[(i)] The assertion follows from the definitions of  $\mathfrak{n}_q(\mu)$ and $\mathfrak{n}_q(\rho)$.

  \item[(ii)] The equations of $\TT_0$ are $x_0=0,x_1=0$ \cite[Lemma 5.2]{DMP_OrbLineMJOM}. Therefore, $K_{\rho,\infty}\in\TT_0$ whereas other points of $\Ls_{\rho}$ do not belong to $\TT_0$, see  \eqref{eq3:K}, \eqref{eq3:Lrho}.

  \item[(iii)] By hypothesis and by \eqref{eq5:roots}, for $\varpi(\Ls_{\rho},\TT_t)=0$ we have the unique root $t=0$, i.e. $\mathfrak{n}_q(\rho)=1$. By Theorem \ref{th23:orbitellmu}(iv)(v) and Lemma \ref{lem5:nqmu}, $\mathfrak{n}_{q}(\mu)\in\{0,2,4\}$. Thus,  $\mathfrak{n}_q(\rho)\ne\mathfrak{n}_{q}(\mu).$

  \item[(iv)] The assertions follow from hypothesis and \eqref{eq5:roots}.
We also use \cite[Section 1.5]{Hirs_PGFF}. \qedhere
\end{description}
\end{proof}

\section{Stabilizers of $\boldsymbol{\Ls_\rho}$-lines and sizes of orbits}\label{sec6}
We denote by $G_q^{\infty}$ the subgroup of  $G_q$  fixing the point $K_{\rho,\infty}=\Pf(0,0,1,0)$. Let $\MM^{\infty}$ be the matrix corresponding to a projectivity of $G_q^{\infty}$.

\begin{lemma}\label{lem6:stabKinf}
The general form of the matrix $\MM^{\infty}$  is as follows:
 \begin{equation}\label{eq6:MMinf}
   \MM^{\infty}=\left[
 \begin{array}{cccc}
 1&0&0&0\\
 0&d&0&0\\
 0&0&d^2&0\\
 0&0&0&d^3
 \end{array}
  \right],~d\in\F_q^*.
\end{equation}
\end{lemma}
 \begin{proof}
We find the version of matrix $\MM$ of \eqref{eq21:M} fixing the point $\Pf(0,0,1,0)$. For $\delta\in\F_q^*$,  $\Pf(0,0,1,0)$ and $\Pf(0,0,\delta,0)$ represent the same point. We have
\begin{equation*}
[0,0,1,0]\times\MM=[3ab^2,b^2c+2abd,ad^2+2bcd,3cd^2]=[0,0,\delta,0],~\delta\in\F_q^*,
\end{equation*}
that implies $3ab^2=0,b^2c+2abd=0,ad^2+2bcd=\delta,3cd^2=0$. If $a=b=0$ then $ad^2+2bcd=0$, contradiction. If $a=0,b\ne0$ then $b^2c=0$ and $2bcd=\delta$, contradiction. So, $a\ne0$, $b=0$. We have, $ad^2=\delta,3cd^2=0$. From $a\ne0$ follows $d\ne0$ and $c=0$. Thus,
\begin{equation*}
  \MM^{\infty}=\left[
 \begin{array}{cccc}
 a^3&0&0&0\\
 0&a^2d&0&0\\
 0&0&ad^2&0\\
 0&0&0&d^3
 \end{array}
  \right],~a,d\in\F_q^*.
\end{equation*}
One may choose $a = 1$, as we consider points in homogeneous coordinates.
\end{proof}

We want to determine the stabilizer group of  $\Ls_{\rho}$ and its orbit under $G_q$. We denote
the subgroup of  $G_q$  fixing $\Ls_{\rho}$ by  $G_q^{\Ls_\rho}$. Let $\MM^{\Ls_\rho}$ be the matrix corresponding to a projectivity of $G_q^{\Ls_{\rho}}$.

\begin{lemma}\label{lem6:stabil0inf}
Let $q$ be even or let $-2\rho$ be a non-cube in $\F_q$. Then  the general form of the matrix $\MM^{\Ls_\rho}$ corresponding to a projectivity of $G_q^{\Ls_\rho}$ is as follows:
 \begin{equation}\label{eq6:MMLrho}
   \MM^{\Ls_\rho}=\left[
 \begin{array}{cccc}
 1&0&0&0\\
 0&d&0&0\\
 0&0&d^2&0\\
 0&0&0&d^3
 \end{array}
  \right],~d\in\F_q^*,~ d  \T{ is a cubic root of unity}.
\end{equation}
\end{lemma}

\begin{proof}
Let a projectivity $\Psi\in G_q^{\Ls_\rho}$. We consider the case $K_{\rho,\infty} \Psi  =  K_{\rho,\gamma}$ for some $\gamma\in\F_q$. The general form of the matrix $\MM$ corresponding to $\Psi$ is given by \eqref{eq21:M}. We have:
\begin{equation*}
[0,0,1,0]\times\MM=[3ab^2,b^2c+2abd,ad^2+2bcd,3cd^2]=[\rho,0,\gamma,1].
\end{equation*}
This implies $ab^2/\rho=cd^2$ and $a,b,c,d \neq 0$.  If $q$ is even, we have also $b^2c = 0$, contradiction.
Now consider the case $q$ odd. As $\MM$ is defined up to a factor of proportionality, we can put $b=1$.  From  $a/\rho=cd^2$ and $c+2ad=0$ we obtain $d^3=-1/2\rho$, contradiction as $-1/2\rho$ (together with $-2\rho$) is not a cube in $\F_q$.

Thus, $K_{\rho,\infty} \Psi\ne K_{\rho,\gamma}$ with $\gamma\in\F_q$. The only possible case is $K_{\rho,\infty} \Psi= K_{\rho,\infty}$, see Lemma~\ref{lem6:stabKinf}.
The matrix $\MM^{\Ls_\rho}$ must be of the same form as $\MM^{\infty}$ \eqref{eq6:MMinf} but the set of possible values of $d$ can be a proper subset of $\F_q^*$.
We should provide  $K_{\rho,0} \Psi  =  K_{\rho,\gamma}$ for some $\gamma\in\F_q$.
As $ [\rho,0,0,1]\times\MM^{\infty} = [\rho,0,0,d^3]$, it can happen only if $d^3=1$.
\end{proof}

\begin{lemma}\label{lem6:CubicRoots}
 Let $q\equiv  -1 \pmod3$. Then the equation $x^3=c$ has a unique solution $c^r$ where $3r+r'(q-1)=1$ and $r,r'$ are integers.
\end{lemma}

\begin{proof}
We use  \cite[Section 1.5(iv)]{Hirs_PGFF}.
\end{proof}

\begin{lemma}\label{lem82:UniqueOrbit_qm1}
 Let $q\equiv  -1 \pmod3$. Then all $\Ls_\rho$ lines belong to the same orbit $\Ob_1$.
\end{lemma}
\begin{proof}
Consider the line $\Ls_\rho$, $\rho \in\F_q^*$. By Lemma \ref{lem6:CubicRoots} there exists $d \in\F_q^*$ such that $d^3 = 1/\rho$. Let
\begin{equation*}
   \MM=\left[
 \begin{array}{cccc}
 1&0&0&0\\
 0&d&0&0\\
 0&0&d^2&0\\
 0&0&0&d^3
 \end{array}
  \right]
\end{equation*}
and let  $\Psi$ be the projectivity corresponding to   $\MM$. Then
\begin{equation*}
[0,0,1,0]\times\MM=[0,0,d^2,0];~[1,0,0,1]\times\MM=[1,0,0,d^3]=[1,0,0,1/\rho].
\end{equation*}
As $\Pf(0,0,d^2,0)=\Pf(0,0,1,0)$ and  $\Pf(1,0,0,1/\rho)=\Pf(\rho,0,0,1)$, it means that $\Ls_1 \Psi= \Ls_\rho$.
\end{proof}

\begin{lemma}\label{lem6:stabLrhoqm1}
 Let $q\equiv  -1 \pmod3$, $q$ odd. Then $G_q^{\Ls_\rho}$ has order 2  and the matrix $\MM^{\Ls_\rho}$ corresponding to the non-trivial projectivity of $G_q^{\Ls_\rho}$ has the form \eqref{eq21:M} with
 \begin{equation*}
   a = \sqrt[3]{1/2\rho},   ~b = 1,  ~ c = \sqrt[3]{2/\rho^2},   ~ d = -\sqrt[3]{1/2\rho}.
 \end{equation*}
\end{lemma}

\begin{proof}
Let a projectivity $\Psi\in G_q^{\Ls_\rho}$ and let $\MM$ be the matrix corresponding to $\Psi$.
If $K_{\rho,\infty} \Psi  =  K_{\rho,\infty} $, we have $\MM=\MM^{\infty}$, see Lemma \ref{lem6:stabKinf} with \eqref{eq6:MMinf}.  Lemma \ref{lem6:CubicRoots} and the proof of Lemma \ref{eq6:MMLrho} imply that in~\eqref{eq6:MMinf} we have $d=1$, so $\Psi$ is the identity projectivity.

Now we consider the case $K_{\rho,\infty} \Psi  =  K_{\rho,\gamma} $ for some $\gamma\in\F_q$. The general form of the matrix $\MM$ corresponding to $\Psi$ is given by \eqref{eq21:M}. We have:
\begin{equation*}
[0,0,1,0]\times\MM=[3ab^2,b^2c+2abd,ad^2+2bcd,3cd^2]=[\rho,0,\gamma,1].
\end{equation*}
This implies $ab^2/\rho=cd^2$ and $a,b,c,d \neq 0$. As $\MM$ is defined up to a factor of proportionality, we can put $b=1$.  From  $a/\rho=cd^2$ and $c+2ad=0$ we obtain $d^3=-1/2\rho$. By Lemma \ref{lem6:CubicRoots} this equation has the unique solution $d=-\sqrt[3]{1/2\rho}$.

Now we consider  $K_{\rho,0} \Psi$. Taking into account $b=1$, the following holds:
\begin{equation}\label{eq6:inf_x_MM}
[\rho,0,0,1]\times\MM = [\rho a^3+1,\rho a^2c+d,\rho ac^2+d^2,\rho c^3-1/2\rho].
\end{equation}
If $K_{\rho,0} \Psi = K_{\rho,\infty}$, then $\rho a^3+1=0$ and $\rho a^2c+d=0$. Lemma \ref{lem6:CubicRoots} implies $a=-\sqrt[3]{1/\rho}$, so $c = -d\sqrt[3]{1/\rho}$. Then $ad-bc=0$, contradiction.
If $K_{\rho,0} \Psi = K_{\rho,\gamma}$ for some $\gamma\in\F_q$, then $(\rho a^3+1)/\rho=\rho c^3-1/2\rho$ and $\rho a^2c+d=0$ from which we obtain: $2\rho^3a^9+3\rho^2a^6-1=0$. Put $t=\rho a^3$. Then we obtain $2t^3+3t^2-1=(t+1)^2(2t-1)=0$. If $t = -1$, then by
Lemma~\ref{lem6:CubicRoots} $a=-1/\sqrt[3]{\rho}$, so $c = -d/\sqrt[3]{\rho}$ and again $ad-bc=0$, contradiction. If $t=1/2$ then by
Lemma \ref{lem6:CubicRoots} $a=\sqrt[3]{1/2\rho}$, so $c = \sqrt[3]{2/\rho^2}$.
Finally, $ad-bc=0$ implies $-\sqrt[3]{1/4\rho^2}=\sqrt[3]{2/\rho^2}$ whence $-1/4\rho^2=2/\rho^2$ and $9/(4\rho) = 0$, contradiction as $q\not\equiv  0 \pmod3$.
\end{proof}

\begin{lemma}\label{lem6:stabil0infqp1size12}
 Let $q\equiv  1 \pmod3$, $q$ odd and let  $-2\rho$ be a cube in $\F_q$. Then $G_q^{\Ls_\rho}$ has order 12  and is isomorphic to the group $\bold{A}_4$. A matrix $\MM^{\Ls_\rho}$  of $G_q^{\Ls_\rho}$ either has the form:
 \begin{equation}\label{eq3:MM ell_0infty}
   \MM^{\Ls_\rho}=\left[
 \begin{array}{cccc}
 1&0&0&0\\
 0&d&0&0\\
 0&0&d^2&0\\
 0&0&0&d^3
 \end{array}
  \right],~d\in\F_q^*,~ d  \T{ is a cubic root of unity},
\end{equation}
or has the form \eqref{eq21:M} with
\begin{equation*}
  a = \T{a cubic root of } 1/2\rho,   ~ b = 1,  ~ c = -d/\rho a^2,   ~ d = \T{a cubic root of } -1/2\rho.
\end{equation*}
\end{lemma}

\begin{proof}
Preliminarily we note that as $3 | (q-1)$, by \cite[Section 1.5 (v)]{Hirs_PGFF} the equation $x^3=c$ has 3 or no solutions in F.
Let a projectivity $\Psi\in G_q^{\Ls_\rho}$ and let $\MM$ be the matrix corresponding to $\Psi$.
If $K_{\rho,\infty} \Psi  =  K_{\rho,\infty}$, we have $\MM=\MM^{\infty}$, see the proof of Lemma \ref{lem6:stabil0inf} and \eqref{eq6:MMinf}.

Now we consider the case $K_{\rho,\infty} \Psi  =  K_{\rho,\gamma}$ for some $\gamma\in\F_q$. The general form of the matrix $\MM$ corresponding to $\Psi$ is given by \eqref{eq21:M}. We have:
\begin{equation*}
[0,0,1,0]\times\MM=[3ab^2,b^2c+2abd,ad^2+2bcd,3cd^2]=[\rho,0,\gamma,1].
\end{equation*}
This implies $ab^2/\rho=cd^2$ and $a,b,c,d \neq 0$. As $\MM$ is defined up to a factor of proportionality, we can put $b=1$.  From  $a/\rho=cd^2$ and $c+2ad=0$ we obtain $d^3=-1/2\rho$.

Now consider  $K_{\rho,0} \Psi$. The relation \eqref{eq6:inf_x_MM} holds.
If $K_{\rho,0} \Psi = K_{\rho,\infty}$, we have $\rho a^3+1=0$ and $\rho a^2c+d=0$. Then $ad-bc=ad-(-d/\rho a^2)=(\rho a^3d+d)/\rho a^2= (-d+d)/\rho a^2=0$, contradiction.
If $K_{\rho,0} \Psi = K_{\rho,\gamma}$ for some $\gamma\in\F_q$, then $(\rho a^3+1)/\rho=\rho c^3-1/2\rho$ and $\rho a^2c+d=0$ from which we obtain: $c = -d/\rho a^2$, $2\rho^3a^9+3\rho^2a^6-1=0$. Put $t=\rho a^3$. Then we obtain $2t^3+3t^2-1=(t+1)^2(2t-1)=0$. If $t = -1$, then  $a^3=-1/\rho$ and again $ad-bc=0$, contradiction. If $t=1/2$ then $a^3=1/2\rho$. By hypothesis, $1/2\rho$ is a cube because it is the product of the two cubes $-1$ and $-1/2\rho$. Finally, $ad-bc=(\rho a^3d+d)/\rho a^2= (d/2+d)/\rho a^2= 3d/2\rho a^2 \neq 0$ if  $q \not\equiv  0 \pmod3$.

By direct computation using Maple${}^\mathrm{TM}$ \cite{Maple}, a non trivial matrix of the form  \eqref{eq3:MM ell_0infty} has order three, whereas of the other nine matrices, the three matrices having $d=-a$ have order two and the other six have order three.
The only group of order $12$ having  three elements of order two and eight elements of order three is $\bold{A}_4$, see \cite{groupbook}.
\end{proof}

\begin{theorem}\label{th6:main}
\begin{description}
  \item[(i)]
Let $q\equiv  1 \pmod3$. Let $q$ be even or let $-2\rho$ be a non-cube in $\F_q$.  Then the size of the subgroup $G_q^{\Ls_\rho}$  of $G_q$ fixing the $\EnG$-line $\Ls_\rho$ is $\#G_q^{\Ls_\rho}=3$. The size of the orbit of $\Ls_\rho$ under $G_q$ is equal to $(q^3-q)/3.$

  \item[(ii)]
Let $q\equiv  1 \pmod3$. Let $q$ be odd and let $-2\rho$ be a cube in $\F_q$.  Then the size of the subgroup $G_q^{\Ls_\rho}$  of $G_q$ fixing the $\EnG$-line $\Ls_\rho$ is $\#G_q^{\Ls_\rho}=12$ and $G_q^{\Ls_\rho}\cong \bold{A}_4$. The size of the orbit of $\Ls_\rho$ under $G_q$ is equal to $(q^3-q)/12.$

  \item[(iii)]
Let $q\equiv  -1 \pmod3$. Let $q$ be even. Then $\#G_q^{\Ls_\rho}=1$ and the size of the orbit of $\Ls_\rho$ under $G_q$ is equal to $q^3-q.$
  \item[(iv)]
Let $q\equiv  -1 \pmod3$. Let $q$ be odd. Then $\#G_q^{\Ls_\rho}=2$ and the size of the orbit of $\Ls_\rho$ under $G_q$ is equal to $(q^3-q)/2.$
\end{description}
\end{theorem}

\begin{proof}
\begin{description}
  \item[(i)]
We take the matrix $\MM^{\Ls_\rho}$.  By  \cite[Section 1.5 (iii)]{Hirs_PGFF}, the equation $d^3=1$ has $3$ solutions if  $q\equiv  1 \pmod3$.
By \cite[Lemma 2.44(ii)]{Hirs_PGFF}, the size of the orbit of $\Ls_\rho$ under $G_q$ is $\#G_q/\#G_q^{\Ls_\rho}=(q^3-q)/3$.
  \item[(ii)]
We apply  Lemma \ref{lem6:stabil0infqp1size12} and \cite[Lemma 2.44(ii)]{Hirs_PGFF}.
  \item[(iii)]
By  \cite[Section 1.5 (ii)]{Hirs_PGFF}, the equation $d^3=1$ has a unique solution if  $q\equiv  -1 \pmod3$. So,  $\#G_q^{\Ls_\rho}=1$ and $\#G_q/\#G_q^{\Ls_\rho}=q^3-q.$
  \item[(iv)]
We apply Lemma \ref{lem6:stabLrhoqm1} and \cite[Lemma 2.44(ii)]{Hirs_PGFF}. \qedhere
\end{description}
\end{proof}

\section{A cubic equation and incidence matrices, even $\boldsymbol{q}$}\label{sec7:cubEqEven}

We consider the cubic equation regarding $t$:
\begin{align}
& \widetilde{F}_{\rho,\gamma}(t)=t^3+\gamma t^2+\rho=0,~\gamma\in\F_q,~\rho\in\F_q^*, ~q\T{ is even}.\label{eq71:cubeqrho}
\end{align}
For $\widetilde{F}_{\rho,\gamma}(t)$, the discriminant $\widetilde{\Delta}$ and the invariant $\widetilde{\delta}$, defined in \cite[Section 1.8, equation (1.15), Lemma 1.18]{Hirs_PGFF}, are as follows:
\begin{align}
 &\widetilde{\Delta}=\rho^2\ne0,~\widetilde{\delta}=\frac{\gamma^3}{\rho}+1.\label{eq71:discrimtildeeven}
   \end{align}
Let $q$ be even. Let $\mathrm{Tr}_2(\eta)$ be the absolute trace of an element $\eta\in\F_q$. For the fixed $\rho\in\F_q^*$, we denote
 \begin{align}\
&\widetilde{\Wb}_q(\rho)\triangleq\#\left\{\gamma\;|\;\mathrm{Tr}_2\left(\frac{\gamma^3}{\rho}+1\right)=1,~\gamma\in\F_q,~
q=2^c\right\}.\label{eq71:Wtilde}
   \end{align}
   We denote $\widetilde{\Nb}_m(\rho)$  the number of $\gamma$ such that the equation $\widetilde{F}_{\rho,\gamma}(t)$ has exactly $m$ distinct solutions $t$ in $\F_q,~m=0,1,2,3$. As $\gamma\in\F_q$, we have
   \begin{equation}\label{eq71:sum of roots}
    \widetilde{\Nb}_0(\rho)+\widetilde{\Nb}_1(\rho)+\widetilde{\Nb}_2(\rho)+\widetilde{\Nb}_3(\rho)=q.
   \end{equation}

\begin{lemma}\label{lem71:N1&N2}
  Let $q$ be even. Let $\widetilde{\delta}$ be as in \eqref{eq71:discrimtildeeven}. Let $\widetilde{\Wb}_q(\rho)$ be as in \eqref{eq71:Wtilde}. For the equation $\widetilde{F}_{\rho,\gamma}(t)$ \eqref{eq71:cubeqrho}, the following holds:
  \begin{description}
    \item[(i)] $\widetilde{F}_{\rho,\gamma}(t)$ has exactly one root in $\F_q$ if and only if the absolute trace   $\mathrm{Tr}_2(\widetilde{\delta})=1$.
In other words,
\begin{equation}\label{eq71:N1(mu)}
 \widetilde{\Nb}_1(\rho)=\widetilde{\Wb}_q(\rho).
\end{equation}

    \item[(ii)] For all admissible $\gamma,\rho$, it is not possible that the equation $\widetilde{F}_{\rho,\gamma}(t)$ has exactly two roots  in $\F_q$, i.e.
    \begin{align}\label{eq71:roots}
    &  \widetilde{\Nb}_2(\rho)=0,~    \widetilde{\Nb}_0(\rho)+\widetilde{\Nb}_1(\rho)+\widetilde{\Nb}_3(\rho)=q.
    \end{align}
  \end{description}
\end{lemma}

\begin{proof}
\begin{description}
    \item[(i)]
  We use \eqref{eq71:discrimtildeeven} and \cite[Corollary 1.15(ii)]{Hirs_PGFF}.
    \item[(ii)] By  \eqref{eq71:discrimtildeeven}, $\widetilde{\Delta}\ne0$. In this case, by \cite[Theorem 1.34]{Hirs_PGFF}, the corresponding cubic equation cannot have exactly two roots in $\F_q$. Finally, we use \eqref{eq71:sum of roots}. \qedhere
  \end{description}
\end{proof}


\begin{lemma}\label{lem71:q2m-1}
 Let $q=2^{2m-1}\equiv-1\pmod3$, $m\ge2$. We have
  \begin{equation}\label{eq71:q=2 2m-1}
   \widetilde{\Wb}_{2^{2m-1}}(\rho)=\frac{q}{2}=2^{2m-2},~\forall\rho\in\F_q^*.
  \end{equation}
\end{lemma}

\begin{proof}
  For a fixed $\rho$, when $\gamma$ runs over $\F_q$ with $q=2^{2m-1}\equiv-1\pmod3$, the values of  $\gamma^3/\rho+1$, also run over $\F_q$. Half of the field elements have absolute trace one.
\end{proof}

\begin{lemma}\label{lem71:q2m}
 Let $q=2^{2m}\equiv1\pmod3$, $m\ge2$. We have
  \begin{equation}\label{eq71:q=2 2m}
   \widetilde{\Wb}_{2^{2m}}(\rho)=\left\{\begin{array}{lcl@{}}
                                    2^{2m-1}+(-2)^m =\frac{1}{2}q+(-1)^m\sqrt{q}& \T{if} &\rho\T{ is a cube in }\F_q \\
                                   2^{2m-1}+(-2)^{m-1}=\frac{1}{2}q-(-1)^m\cdot\frac{1}{2}\sqrt{q} & \T{if}&\rho\T{ is a non-cube in }\F_q
                                  \end{array}
   \right..
  \end{equation}
\end{lemma}

\begin{proof}
 In  \cite{Carlitz}, for a field $\F_{q}$, $q=p^n$, $p$ prime, an exponential sum
 \begin{equation*}
   S(a,0)\triangleq\sum_{x\in\F_q}exp\left(\frac{2\pi i}{p}\mathrm{Tr}_p(ax^3)\right), a\in \F_q\T{ is a constant},
 \end{equation*}
   is considered, where $\mathrm{Tr}_p(ax^3)$ is the absolute trace of $ax^3$. For $q=2^n$, in the literature, see e.g. \cite{FanWangXu,ZinHellesCharp} and the references therein,
   this sum is presented in the form $\sum_{x\in\F_q}(-1)^{\mathrm{Tr}_2(ax^3)}$. For $q=2^{2m}$, in \cite{Carlitz}, it is proved:
   \begin{equation}\label{eq71:S(a,0)}
     S(a,0)=\sum_{x\in\F_q}(-1)^{\mathrm{Tr}_2(ax^3)}=\left\{\begin{array}{lcl@{}}
                                    (-1)^{m+1}2^{m+1} & \T{if} &a\T{ is a cube in }\F_q \\
                                   (-1)^{m}2^{m} & \T{if}&a\T{ is a non-cube in }\F_q
                                  \end{array}
   \right..
   \end{equation}

 For $q=2^{2m}$, we have $\mathrm{Tr}_2(1) =  0$ that implies
 \begin{align*}
&\widetilde{\Wb}_{2^{2m}}(\rho)=\#\left\{\gamma\;|\;\mathrm{Tr}_2\left(\frac{\gamma^3}{\rho}\right)=1,~\gamma\in\F_q,~
q=2^{2m}\right\}.
   \end{align*}
   We denote
 \begin{align*}
&\widetilde{\widetilde{\Wb}}_{2^{2m}}(\rho)\triangleq\#\left\{\gamma\;|\;\mathrm{Tr}_2\left(\frac{\gamma^3}{\rho}\right)=0,~\gamma\in\F_q,~
q=2^{2m}\right\}.
   \end{align*}
   Obviously, $\widetilde{\widetilde{\Wb}}_{2^{2m}}(\rho)+\widetilde{\Wb}_{2^{2m}}(\rho)=2^{2m}$ and $\widetilde{\widetilde{\Wb}}_{2^{2m}}(\rho)-\widetilde{\Wb}_{2^{2m}}(\rho)=S(1/\rho,0)$, that gives
   \begin{equation}\label{eq71:widetilde Wb}
     \widetilde{\Wb}_{2^{2m}}(\rho)=2^{2m-1}-\frac{1}{2}S\left(\frac{1}{\rho},0\right).
   \end{equation}
  If $\rho$ is a cube (resp. a non-cube) in $\F_q$ then $1/\rho$ also is a cube (resp. a non-cube). Therefore the assertion \eqref{eq71:q=2 2m} follows from \eqref{eq71:S(a,0)}, \eqref{eq71:widetilde Wb}.
\end{proof}

\begin{remark}
 The 1-st row of \eqref{eq71:q=2 2m} follows from the context of \cite[Section 4]{CerPavDM}. This is noted in \cite[equation (4.3)] {DMP_IncO6JG}. In \cite{CerPavDM} the results of \cite{WolfNumbSolut} are used.
\end{remark}

\begin{lemma}\label{lem72:cubic_eqrho}
Let $q$ be even. Let $\gamma,t\in\F_q$. Let the point $K_{\rho,\gamma}=\Pf(\rho,0,\gamma,1)$ belong to the osculating plane $\pi_\T{osc}(t)$. Then the values of $\rho,\gamma,$ and $t$ satisfy the cubic equation $\widetilde{F}_{\rho,\gamma}(t)$ of \eqref{eq71:cubeqrho}.
\end{lemma}

 \begin{proof}
 For even $q$, we have $\pi_\T{osc}(t)=\boldsymbol{\pi}(1,t,t^2,t^3)$, $t\in\F_q$, that implies the assertion.
   \end{proof}

  \begin{theorem}
 Let $q$ be even. For the orbit $\Ob_\rho$, generated by a line $\Ls_\rho$, the following holds.
  \begin{align}\label{eq72:solut&incidrho}
 &\Pb_{\Tr}=1,~\Pb_{0_\Gamma}=\widetilde{\Nb}_{0}(\rho),~\Pb_{1_\Gamma}=
 \widetilde{\Nb}_{1}(\rho)=\widetilde{\Wb}_q(\rho),~\Pb_{3_\Gamma}=\widetilde{\Nb}_{3}(\rho).
      \end{align}
\end{theorem}

\begin{proof}
By Lemma \ref{lem72:cubic_eqrho}, if, for a fixed $\gamma$,  the equation $\widetilde{F}_{\rho,\gamma}(t)$ \eqref{eq71:cubeqrho} has exactly $m$ distinct solutions $t$ in $\F_q$ then the point $K_{\rho,\gamma}$ belongs to exactly $m$ distinct osculating planes. So, the set $\Ls_{\rho}\setminus\{K_{\rho,\infty}\}$ contains $\widetilde{\Nb}_m(\rho)$ points belonging to exactly $m$ distinct osculating planes. In particular, if $m=2$, they are $\Tr$-points, see Lemma \ref{lem4:Tpoint=2Gammapoint}. But, by Lemma \ref{lem71:N1&N2}(ii), $\widetilde{\Nb}_2(\rho)=0$ for all $\rho$. Also, by Theorem \ref{th5:Lrho&tang}, we have on $\Ls_{\rho}$ one $\Tr$-point $K_{\rho,\infty}$. Finally, we use Lemma~\ref{lem4:orb the same}.
\end{proof}

  \begin{theorem}\label{th72:incidrho2m-1}
 Let $q$ be even. Let $\widetilde{\Wb}_q(\rho)$ be as in \eqref{eq71:Wtilde}, \eqref{eq71:q=2 2m-1}, \eqref{eq71:q=2 2m}. Let the orbit $\Ob_\rho$ be  generated by a line $\Ls_\rho$. Then, for the point-line incidence
matrix corresponding to the orbit the following holds:

 Let $q=2^{2m-1}\equiv  -1\pmod3$. Then $\#\Ob_\rho=q^3-q$ for all $\rho$; $\widetilde{\Wb}_q(\rho)=q/2$,  and
\begin{align}
&\Pb_{\Tr}=1,~\Lb_{\Tr}=q-1;~\Pb_{1_\Gamma}=\frac{q}{2},~\Lb_{1_\Gamma}=q;\db\label{eq7:2 2m-1}\\
&\Pb_{3_\Gamma}=\frac{q-2}{6},~\Lb_{3_\Gamma}=q-2;~\Pb_{0_\Gamma}=\frac{q+1}{3},~\Lb_{0_\Gamma}=q+1.\notag
 \end{align}

 Let $q=2^{2m}\equiv1\pmod3$. Then $\#\Ob_\rho=\frac{1}{3}(q^3-q)$ for all $\rho$, $\widetilde{\Wb}_q(\rho)$ is as in \eqref{eq71:q=2 2m}, and
\begin{align}
&\Pb_{\Tr}=1,~\Lb_{\Tr}=\frac{1}{3}(q-1),~
\Pb_{1_\Gamma}=\widetilde{\Wb}_q(\rho),~\Lb_{1_\Gamma}=\frac{2}{3}\widetilde{\Wb}_q(\rho),\db\label{eq7:2 2m}\\
&\Pb_{3_\Gamma}=\frac{q-1-\widetilde{\Wb}_q(\rho)}{3},~\Lb_{3_\Gamma}=\frac{2(q-1-\widetilde{\Wb}_q(\rho))}{3},~
\Pb_{0_\Gamma}=\Lb_{0_\Gamma}=\frac{2q-2\widetilde{\Wb}_q(\rho)+1}{3}.\notag
 \end{align}

The plane-line incidence matrix contains, according to \eqref{eq4:=0}--\eqref{eq4:Lb=Lambda}, the same values of the point-line incidence matrix, but in this
case they refer to $\Pi_\pi, \Lambda_\pi$ instead of $\Pb_\pk, \Lb_\pk$.
\end{theorem}

\begin{proof}
For both the cases \eqref{eq7:2 2m-1} and \eqref{eq7:2 2m}, by \eqref{eq72:solut&incidrho}, we have $\Pb_{\Tr}=1$, $\Pb_{1_\Gamma}=\widetilde{\Nb}_{1}(\rho)=\widetilde{\Wb}_q(\rho)$. Also, we use \eqref{eq4:obtainLamb1} to obtain $\Lb_{\pk}$ and take the sizes of orbits $\M_\pk$ from \eqref{eq21:pi(pk)} and the orbits $\Ob_\rho$ from Theorem \ref{th6:main}.

For \eqref{eq7:2 2m-1}, by Lemma \ref{lem71:q2m-1}, $\widetilde{\Wb}_q(\rho)=q/2$. Now, from~\eqref{eq4:EnGorbit3}, we have  $\Pb_{3_\Gamma}=(q-2)/6=\widetilde{\Nb}_{3}(\rho)$.  Then, by \eqref{eq71:roots}, we obtain  $\widetilde{\Nb}_{0}(\rho)=(q+1)/3=\Pb_{0_\Gamma}$.

For \eqref{eq7:2 2m}, from~\eqref{eq4:EnGorbit3}, we have  $\Pb_{3_\Gamma}=(q-1-\widetilde{\Wb}_q(\rho))/3=\widetilde{\Nb}_{3}(\rho)$.  Then, by \eqref{eq71:roots}, we obtain  $\widetilde{\Nb}_{0}(\rho)=(2q-2\widetilde{\Wb}_q(\rho)+1)/3=\Pb_{0_\Gamma}$.

The last assertion follows from Theorem \ref{th4:the same}.
\end{proof}

\section{Orbits  $\boldsymbol{\Ob_\rho}$, even $\boldsymbol{q}$}\label{sec8:orbits even}
\begin{corollary}\label{cor72:two dist orb}
 Let $q=2^{2m}\equiv1\pmod3$. Let $\rho'$ be a cube in $\F_q$ whereas $\rho''$ be a non-cube. Then the $\frac{1}{3}(q^3-q)$-orbits $\Ob_{\rho'}$ and $\Ob_{\rho''}$ generated by  $\Ls_{\rho'}$- and $\Ls_{\rho''}$-lines, respectively, are distinct.
\end{corollary} 

\begin{proof}
  The assertion follows from Lemma \ref{lem71:q2m} and Theorem \ref{th72:incidrho2m-1}.
\end{proof}


\begin{theorem}\label{th72:diforbLs}
 \begin{description}
  \item[(i)]
Let $q=2^{2m}\equiv1\pmod3$.  Two lines $\Ls_{\rho'}$ and $\Ls_{\rho''}$ of type \eqref{eq3:Lrho} belong to different orbits of $G_q$ if and only if $\log \rho'\not\equiv \log \rho''\pmod3$.

  \item[(ii)]
Let $q=2^{2m-1}\equiv  -1\pmod3$. Then two lines $\Ls_{\rho'}$ and $\Ls_{\rho''}$ of type \eqref{eq3:Lrho} always belong to the same orbit of $G_q$.
\end{description}
\end{theorem}

\begin{proof}
    \begin{description}
    \item[(i)]   Let $\Psi\in G_q$ be a projectivity such that $\Ls_{\rho'}\Psi =\Ls_{\rho''}$.
Suppose $K_{\rho',\infty} \Psi  =  K_{\rho'',\gamma}$ for some $\gamma\in\F_q$. The general form of the matrix $\MM^\Psi$ corresponding to $\Psi$ is given by \eqref{eq21:M}. We have for $q$ even:
\begin{equation*}
[0,0,1,0]\times\MM^\Psi=[ab^2,b^2c,ad^2,cd^2]=[\rho'',0,\gamma,1].
\end{equation*}
This implies $c,d \neq 0$, $b=0$ and $\rho''=0$, contradiction.

 Thus, $K_{\rho',\infty} \Psi= K_{\rho'',\infty}$, i.e. $\Pf(0,0,1,0)\Psi=\Pf(0,0,\gamma,0)$, that implies  $a,d \neq 0$, $b=c=0$, so the matrix $\MM^{\Psi}$ must be of the same form as $\MM^{\infty}$ \eqref{eq6:MMinf} but the set of possible values of $d$ can be a proper subset of $\F_q^*$.
We should provide  $K_{\rho',0} \Psi  =  K_{\rho'',\gamma}$ for some $\gamma\in\F_q$, i.e. $[\rho',0,0,1]\MM^{\Psi}=[\rho'',0,\gamma,1]$.
As $ [\rho',0,0,1]\times\MM^{\infty} = [\rho',0,0,d^3]$, it can happen only if $d^3=\rho'/\rho''$. \\
By Lemma \ref{lem82:logofcubes}, $\rho'/\rho''$ is a cube if and only if log $(\rho'/\rho'') \equiv0\pmod3$ that happens if and only if log $\rho' \equiv$ log $\rho'' \pmod3$.
    \item[(ii)]  We use Lemma \ref{lem82:UniqueOrbit_qm1}.  \qedhere
\end{description}
\end{proof}

\begin{corollary}
  Let $q$ be even.
  \begin{description}
    \item[(i)] Let $\alpha$ be a primitive element of $\F_q$.  If $q=2^{2m}\equiv1\pmod3$, there are three distinct $\frac{1}{3}(q^3-q)$-orbits generated by $\Ls_\rho$-lines with $\rho=\alpha^j$, $j=0,1,-1$, respectively.
    \item[(ii)] If $q=2^{2m-1}\equiv  -1\pmod3$, all $\Ls_\rho$-lines generate the same $(q^3-q)$-orbit.
    \item[(iii)] All the orbits generated by $\Ls_\rho$-lines are different from the ones generated by $\ell_\mu$-lines.
  \end{description}
\end{corollary}

\begin{proof}
  The assertions follow from Theorems \ref{th5:Lrho&tang}, \ref{th6:main}, and \ref{th72:diforbLs}.
\end{proof}

\section{A cubic equation and incidence matrices, odd $\boldsymbol{q}$}\label{sec8:cubEqIncOdd q}

We consider a cubic equation regarding $t$.
\begin{align}\label{eq81:cubEqodd}
 F_{\rho,\gamma}(t)=t^3-3\gamma t^2-\rho=0,~\gamma,t,\rho\in\F_q^*,~ q\not\equiv0\pmod3.
\end{align}
For $F_{\rho,\gamma}(t)$, the discriminant $\Delta$ and the Hessian $H(T)$, its coefficients $A_i$ and roots $\beta_{1,2}$, defined in \cite[Section 1.8, equation (1.14), Lemma 1.18, Theorem 1.28]{Hirs_PGFF}, are as follows:
\begin{align}\label{eq81:discrim}
& \Delta=-27\rho(4\gamma^3+\rho),~\gamma,\rho\in\F_q^*;~H(T)=A_0T^2+A_1T+A_2;\db\\
& A_0=-9\gamma^2,~A_1=-9\rho,~A_2=9\rho\gamma,~\beta_{1,2}=\frac{-\rho\pm\rho\sqrt{1+4\rho^{-1}\gamma^3}}{2\gamma^2},~\gamma,\rho\in\F_q^*.\notag
   \end{align}
  We denote $\Nb_m(\rho)$  the number of $\gamma\in\F_q^*$ such that the equation $F_{\rho,\gamma}(t)$ \eqref{eq81:cubEqodd}  has exactly $m$ distinct solutions $t$ in $\F_q^*,~m=0,1,2,3$.

\begin{lemma}\label{lem81:1root}
Let $q\equiv\xi\pmod3$. The equation $F_{\rho,\gamma}(t)$ \eqref{eq81:cubEqodd} has exactly $1$ root $t$ in $\F_q$ if and only if $4\gamma^3+\rho\ne0$ and, also,
$1+4\rho^{-1}\gamma^3$ is a square \emph{(}resp. non-square\emph{)} in $\F_q$ for $\xi=-1$ \emph{(}resp. $\xi=1$\emph{)}.
\end{lemma}

\begin{proof}
  By \cite[Corollary 1.30]{Hirs_PGFF}, for $\Delta=0$, $F_{\rho,\gamma}(t)$ has 1 root in $\F_q$ if all $A_i=0$. But $A_i\ne0$, see \eqref{eq81:discrim}. By \cite[Theorem 1,34, Table 1.3]{Hirs_PGFF}, for $\Delta\ne0$, $F_{\rho,\gamma}(t)$ has 1 root  in $\F_q$, if $H(T)$ has roots or not in $\F_q$ according to $\xi=-1$ or $\xi=1$, respectively.
\end{proof}

For $\beta\in\F_q$, let $\eta(\beta)$ be as in \eqref{eq5:quadr-char}.   We denote
\begin{align}\label{eq81:Nk}
 & \Nk_{q,\rho}\triangleq\#\{\gamma\,|\,\gamma\in\F_q^*,~\eta(1+4\rho^{-1}\gamma^3)=-1\},~q\equiv1\pmod3.
\end{align}

\begin{lemma}\label{lem81:N1}
\begin{description}
    \item[(i)] Let $q\equiv-1\pmod3$ be odd. Then $\Nb_1(\rho)=(q-3)/2$.

    \item[(ii)]  Let $q\equiv1\pmod3$ be odd. Then $\Nb_1(\rho)=\Nk_{q,\rho}$.
  \end{description}

\end{lemma}
\begin{proof}
 \begin{description}
    \item[(i)]
If $\gamma$ runs over $\F_q^*\setminus\{-\sqrt[3]{\rho/4}\}$ then $\Delta\ne0$ and $1+4\rho^{-1}\gamma^3$ runs over  $\F_q^*\setminus\{1\}$ where there are exactly $(q-1)/2$ non-squares and $(q-3)/2$ non-zero squares. Now we use Lemma \ref{lem81:1root}.

    \item[(ii)] If $\Delta=0$ then $\rho=-4\gamma^3$ and $\eta(1+4\rho^{-1}\gamma^3)=0$. So, the case $\Delta=0$ does not influence the value $\Nk_{q,\rho}$. Now the assertion follows from Lemma \ref{lem81:1root}.\qedhere
  \end{description}
\end{proof}

\begin{lemma}\label{lem82:cubEq}
Let $q$ be odd. Let $\gamma,t\in\F_q^*$. Let the point $K_{\rho,\gamma}=\Pf(\rho,0,\gamma,1)$  belong to the osculating plane $\pi_\T{osc}(t)$. Then the values of $\rho,\gamma,t$ satisfy the cubic equation $F_{\rho,\gamma}(t)$~\eqref{eq81:cubEqodd}.
\end{lemma}

 \begin{proof}
We have $\pi_\T{osc}(t)=\boldsymbol{\pi}(1,-3t,3t^2,-t^3)$, $t\in\F_q$, that implies the assertions.
   \end{proof}

 \begin{theorem}\label{th82:PT P1G}
 Let $q$ be odd. Let $\Nk_{q,\rho}$ be as in \eqref{eq81:Nk}. For the orbit $\Ob_\rho$, generated by a line $\Ls_\rho$, the following holds.
  \begin{align*}
 &\Pb_{\Tr}=\left\{\begin{array}{ccl}
                     2 & \T{if} & q\equiv-1\pmod3 \\
                     1 & \T{if} & q\equiv1\pmod3,~-2\rho\T{ is a non-cube in }\F_q \\
                     4 & \T{if} & q\equiv1\pmod3,~-2\rho\T{ is a cube in }\F_q
                   \end{array}
 \right.;\db\\
 &\Pb_{1_\Gamma}= \Nb_{1}(\rho)+1=(q-1)/2 \T{ if }  q\equiv-1\pmod3;\db \\
  &\Pb_{1_\Gamma}=\Nb_{1}(\rho)= \Nk_{q,\rho} \T{ if } q\equiv1\pmod3.
      \end{align*}
\end{theorem}

\begin{proof}
The values of $\Pb_{\Tr}$ are taken from Theorem \ref{th5:Lrho&tang}.

By Lemma \ref{lem82:cubEq}, if, for a fixed $\gamma$,  the equation $F_{\rho,\gamma}(t)$ \eqref{eq81:cubEqodd} has exactly $m$ distinct solutions $t$ in $\F_q$ then the point $K_{\rho,\gamma}$, $\gamma\in\F_q^*$,  belongs to exactly $m$ distinct osculating planes. So, the set $\Ls_{\rho}\setminus\{K_{\rho,0},K_{\rho,\infty}\}$ contains $\Nb_1(\rho)$ points belonging to exactly $1$ osculating plane.
 Also, by Lemma \ref{lem3:pi0inf}, $K_{\rho,\infty}$ belongs to exactly two osculating planes, $K_{\rho,0}$ is a $1_\Gamma$-point if and only if $q\equiv-1\pmod3$. For $\Nb_{1}(\rho)$ we apply Lemma \ref{lem81:N1}.
\end{proof}

 \begin{theorem}\label{th82:incidOdd}
 Let $q$ be odd. Let $\Nk_{q,\rho}$ be as in \eqref{eq81:Nk}. Let the orbit $\Ob_\rho$ be  generated by a line $\Ls_\rho$. Then, for the point-line incidence
matrix corresponding to the orbit the following holds:

     Let $q\equiv-1\pmod3$. Then $\#\Ob_\rho=(q^3-q)/2$ for all $\rho$ and we have
   \begin{align}\label{eq82:incid-1}
 &  \Pb_{\Tr}=2,~\Lb_{\Tr}=q-1;~\Pb_{1_\Gamma}=\Lb_{1_\Gamma}=\frac{q-1}{2};
\db\\
 &\Pb_{3_\Gamma}=
 \frac{q-5}{6},~\Lb_{3_\Gamma}=\frac{q-5}{2};~\Pb_{0_\Gamma}=\frac{q+1}{3},~\Lb_{0_\Gamma}=\frac{q+1}{2}.\notag
      \end{align}

    Let $q\equiv1\pmod3$. Let $-2\rho$ be a non-cube in $\F_q$. Then $\#\Ob_\rho=(q^3-q)/3$  and
   \begin{align}\label{eq82:incid=1 -2rho non-cube}
 &  \Pb_{\Tr}=1,~\Lb_{\Tr}=\frac{q-1}{3};~\Pb_{1_\Gamma}=\Nk_{q,\rho},~\Lb_{1_\Gamma}=\frac{2}{3}\Nk_{q,\rho};~
\Pb_{3_\Gamma}=
 \frac{q-1-\Nk_{q,\rho}}{3},\db\\
 &\Lb_{3_\Gamma}= \frac{2(q-1-\Nk_{q,\rho})}{3};~\Pb_{0_\Gamma}=\Lb_{0_\Gamma}=\frac{2q+1-2\Nk_{q,\rho}}{3}.\notag
      \end{align}

     Let $q\equiv1\pmod3$. Let $-2\rho$ be a cube in $\F_q$. Then $\#\Ob_\rho=(q^3-q)/12$  and
   \begin{align}\label{eq82:incid=1 -2rho cube}
 &  \Pb_{\Tr}=4,~\Lb_{\Tr}=\frac{q-1}{3};~\Pb_{1_\Gamma}=\Nk_{q,\rho},~\Lb_{1_\Gamma}=\frac{1}{6}\Nk_{q,\rho};~\Pb_{3_\Gamma}=
 \frac{q-7-\Nk_{q,\rho}}{3},\db\\
 &\Lb_{3_\Gamma}= \frac{q-7-\Nk_{q,\rho}}{6};~\Pb_{0_\Gamma}=\frac{2(q-1-\Nk_{q,\rho})}{3},~\Lb_{0_\Gamma}=\frac{q-1-\Nk_{q,\rho}}{6}.\notag
      \end{align}

The plane-line incidence matrix contains, according to \eqref{eq4:=0}--\eqref{eq4:Lb=Lambda}, the same values of the point-line incidence matrix, but in this
case they refer to $\Pi_\pi, \Lambda_\pi$ instead of $\Pb_\pk, \Lb_\pk$.
\end{theorem}

\begin{proof}
The sizes $\#\Ob_\rho$ are taken from Theorem \ref{th6:main}. For all the cases \eqref{eq82:incid-1}--\eqref{eq82:incid=1 -2rho cube}, in the beginning, we take the values $\Pb_\Tr$ and $\Pb_{1_\Gamma}$ from Theorem \ref{th82:PT P1G}. Then, by \eqref{eq4:EnGorbit3}, we obtain $\Pb_{3_\Gamma}$ and $\Pb_{0_\Gamma}$.  Finally, we apply \eqref{eq4:obtainLamb1} to calculate $\Lb_\pk$. We take the sizes  $\#\M_\pk$, $\pk\in\Mk\setminus\{\C\}$, from Theorem \ref{th21:Hirs}(ii)(iii). The last assertion follows from Theorem \ref{th4:the same}.
\end{proof}

\begin{corollary}
 Let $q\equiv1\pmod3$ be odd. Let $\Nk_{q,\rho}$ be as in \eqref{eq81:Nk}. Then $3|\Nk_{q,\rho}$. Moreover, $6|\Nk_{q,\rho}$  if  $-2\rho$ is a cube in $\F_q$. 
\end{corollary}

\begin{proof}
  As the value $\Lb_{1_\Gamma}$ must be an integer, the assertions follow from \eqref{eq82:incid=1 -2rho non-cube} and \eqref{eq82:incid=1 -2rho cube}.
\end{proof}

\section{Orbits $\boldsymbol{\Ob_\rho}$, odd $\boldsymbol{q}$}\label{sec10:orbits odd q}

\begin{lemma}\label{lem82:a}
Let $q\equiv1\pmod3$ be odd. Let $\Rk_m$ be as in \eqref{eq4:Rmdefin}.
  \begin{description}
    \item[(i)] The values of $\rho\in\F_q^*$ can be partitioned into three classes $\Rk_0,\Rk_1,\Rk_2$ such that
    \begin{align*}
    &  \#\Rk_m=\frac{q-1}{3},~m=0,1,2.
    \end{align*}

    \item[(ii)] Let $\log(-2)\equiv\psi\pmod3$, $\psi\in\{0,1,2\}$.
    We have the following.

     There exist two classes $\Rk_m$ such that $-2\rho$ is a non-cube in $\F_q$ for $\rho\in\Rk_m$:
    \begin{align*}
      &\Rk_1,\Rk_2  \T{ if }\psi=0; ~\Rk_0,\Rk_1   \T{ if }\psi=1 ;~\Rk_0,\Rk_2   \T{ if }\psi=2.
    \end{align*}

     There exists one class  $\Rk_m$ such that $-2\rho$ is a cube in $\F_q$  for $\rho\in\Rk_m$:
    \begin{align*}\label{eq82:-2cube}
      &\Rk_0 \T{ if }\psi=0; ~\Rk_2   \T{ if }\psi=1 ;~\Rk_1   \T{ if }\psi=2.
    \end{align*}
  \end{description}
\end{lemma}

\begin{proof}
  The case (i) is obvious. The case (ii) follows from Lemma \ref{lem82:logofcubes}.
\end{proof}

\begin{lemma}\label{lem82:equivsamelog}
Let $q$ be odd. Let $\rho_1, \rho_2 \in \F_q^*$. If $\log \rho_1 \equiv \log \rho_2 \pmod3$, i.e. $\rho_1, \rho_2$ belong to the same class $\Rk_m$, then the lines   $\Ls_{\rho_1}$, $\Ls_{\rho_2}$ belong to the same orbit.
\end{lemma}

\begin{proof}
Consider $\rho_1/\rho_2$: it is a cube as $\log(\rho_1/\rho_2) = \log\rho_1 - \log\rho_2$. By hypothesis, $(\log\rho_1 - \log\rho_2)\pmod 3=0$.
Let $d \in\F_q^*$ such that $d^3 = \rho_1/\rho_2$, let
\begin{equation*}
   \MM=\left[
 \begin{array}{cccc}
 1&0&0&0\\
 0&d&0&0\\
 0&0&d^2&0\\
 0&0&0&d^3
 \end{array}
  \right]
\end{equation*}
and let  $\Psi$ be the projectivity corresponding to   $\MM$. Then
\begin{equation*}
[0,0,1,0]\times\MM=[0,0,d^2,0];~[\rho_1,0,0,1]\times\MM=[\rho_1,0,0,d^3]=[1,0,0,d^3/\rho_1]=[1,0,0,1/\rho_2].
\end{equation*}
As $\Pf(0,0,d^2,0)=\Pf(0,0,1,0)$ and  $\Pf(1,0,0,1/\rho_2)=\Pf(\rho_2,0,0,1)$, it means that $\Ls_{\rho_1} \Psi= \Ls_{\rho_2}$.
\end{proof}

\begin{lemma}\label{lem82:two orbits}
Let $q \equiv1\pmod3$ be odd.
  \begin{description}
    \item[(i)]
Let $\Rk_{m'}$ and $\Rk_{m''}$ be the two classes of values of $\rho$ such that, in accordance with Lemma \emph{\ref{lem82:a}(ii)}, $-2\rho$ is a non-cube in $\F_q$ for $\rho\in\Rk_{m'}\cup\Rk_{m''}$. Let $\rho'\in \Rk_{m'}$, $\rho''\in \Rk_{m''}$. Then the lines $\Ls_{\rho'}$ and $\Ls_{\rho''}$ generate two distinct $\frac{1}{3}(q^3-q)$-orbits $\Ob_{\rho'}$ and $\Ob_{\rho''}$, respectively, every of which contains $(q-1)/3$ lines $\Ls_\rho$ with $\rho$ belonging to the same class $\Rk_{m^\bullet}$.
    \item[(ii)]
Let $\Rk_{m'''}$  be the class of values of $\rho$ such that, in accordance with Lemma \emph{\ref{lem82:a}(ii)}, $-2\rho$ is a cube in $\F_q$ for $\rho\in\Rk_{m'''}$. Let $\rho'''\in \Rk_{m'''}$. Then $m'''\notin\{m',m''\}$ and the line $\Ls_{\rho'''}$ generates the $\frac{1}{12}(q^3-q)$-orbit $\Ob_{\rho'''}$, containing $(q-1)/3$ lines $\Ls_\rho$with $\rho\in\Rk_{m'''}$.
  \end{description}
\end{lemma}

\begin{proof}
  \begin{description}
    \item[(i)]
  We do similarly to Proof of Theorem \ref{th72:diforbLs}. Let $\Psi\in G_q$ be a projectivity such that $\Ls_{\rho'}\Psi =\Ls_{\rho''}$. We consider the case $K_{\rho',\infty} \Psi  =  K_{\rho'',\gamma}$ for some $\gamma\in\F_q$. The general form of the matrix $\MM^\Psi$ corresponding to $\Psi$ is given by \eqref{eq21:M}. We have:
\begin{equation*}
[0,0,1,0]\times\MM^\Psi=[3ab^2,b^2c+2abd,ad^2+2bcd,3cd^2]=[\rho'',0,\gamma,1].
\end{equation*}
This implies $ab^2/\rho''=cd^2$ and $a,b,c,d \neq 0$.
 As $\MM$ is defined up to a factor of proportionality, we can put $b=1$.  From  $a/\rho''=cd^2$ and $c+2ad=0$ we obtain $d^3=-1/2\rho''$, contradiction as $-1/2\rho''$ (together with $-2\rho''$) is not a cube in $\F_q$.

 Thus, $K_{\rho',\infty} \Psi\ne K_{\rho'',\gamma}$ with $\gamma\in\F_q$. The only possible case is $K_{\rho',\infty} \Psi= K_{\rho'',\infty}$, i.e. $\Pf(0,0,1,0)\Psi=\Pf(0,0,\gamma,0)$, see Lemma~\ref{lem6:stabKinf}.
The matrix $\MM^{\Psi}$ must be of the same form as $\MM^{\infty}$ \eqref{eq6:MMinf} but the set of possible values of $d$ can be a proper subset of $\F_q^*$.
We should provide  $K_{\rho',0} \Psi  =  K_{\rho'',\gamma}$ for some $\gamma\in\F_q$, i.e. $[\rho',0,0,1]\MM^{\Psi}=[\rho'',0,\gamma,1]$.
As $ [\rho',0,0,1]\times\MM^{\infty} = [\rho',0,0,d^3]$, it can happen only if $d^3=\rho'/\rho''$. But, $\log(\rho'/\rho'')\not\equiv0\pmod3$ since $m'\ne m''$. So, due to Lemma \ref{lem82:logofcubes},
$\rho'/\rho''$ is not a cube, contradiction. Thus, a projectivity $\Psi\in G_q$ sending  $\Ls_{\rho'}$ to $\Ls_{\rho''}$ does not exist.

Finally, we use Theorem \ref{th6:main}(i) and Lemmas \ref{lem82:a}(i), \ref{lem82:equivsamelog}.
    \item[(ii)]
    We use Theorem \ref{th6:main}(ii) and Lemmas \ref{lem82:a}, \ref{lem82:equivsamelog}. \qedhere
  \end{description}
\end{proof}

\begin{theorem}
 Let $q\equiv1\pmod3$ be odd.
  \begin{description}
    \item[(i)] Let $\rho_1\ne\rho_2$. Then two lines $\Ls_{\rho_1}$ and $\Ls_{\rho_2}$ belong to distinct orbits under $G_q$ if and only if $\log \rho_1\not\equiv\log \rho_2 \pmod3$, i.e. $\rho_1,\rho_2$ belong to distinct classes $\Rk_m$. All $\Ls_\rho$-lines generate three distinct orbits $\Ob_\rho$ every of which contains $(q-1)/3$ $\Ls_\rho$-lines with $\rho$ belonging to the same class $\Rk_m$.

        \item[(ii)] Two orbits $\Ob_{\rho}$, say $\Ob_{\rho}^{(1)}$ and $\Ob_{\rho}^{(2)}$, have size $\frac{1}{3}(q^3-q)$ and are generated by lines $\Ls_{\rho}$ such that $-2\rho$ is a non-cube in $\F_q$, in accordance with Lemma \emph{\ref{lem82:a}(ii)}.  The orbits $\Ob_{\rho}^{(1)}$ and $\Ob_{\rho}^{(2)}$ are different from any orbit~$\Os_\mu$ of \emph{\cite[Section 7]{DMP_OrbLineO6MJOM}}, see also Theorem \emph{\ref{th23:orbitellmu}}.

          \item[(iii)]
            The third orbit $\Ob_{\rho}$, say $\Ob_{\rho}^{(3)}$,  has size $\frac{1}{12}(q^3-q)$ and is generated by a line $\Ls_{\rho}$ such that $-2\rho$ is a cube in $\F_q$, in accordance with Lemma \emph{\ref{lem82:a}(ii)}.

        \item[(iv)] If $q\not\equiv1\pmod{12}$ or $-1/3$ is not a fourth degree  in $\F_q$, i.e. the condition $\Upsilon_{q,\mu}$ \eqref{eq23:Upsilon} does not hold, then the  orbit $\Ob_{\rho}^{(3)}$ is different from any orbit~$\Os_\mu$.

        \item[(v)] If $q\equiv1\pmod{12}$ and $-1/3$ is a fourth degree in $\F_q$, then $\Ob_\rho^{(3)}=\Os_{-1/3}$.
  \end{description}
\end{theorem}

\begin{proof}
  \begin{description}
    \item[(i)]  We use Lemmas \ref{lem82:logofcubes}, \ref{lem82:two orbits}.

    \item[(ii)]  The sizes of the orbits $\Ob_{\rho}^{(1)}$, $\Ob_{\rho}^{(2)}$ follow from Theorem \ref{th6:main}(i). By Theorem \ref{th23:orbitellmu}(iii),  the sizes of the orbits $\Ob_\rho^{(1)}$ , $\Ob_\rho^{(2)}$ and the orbits $\Os_\mu$ are distinct.

    \item[(iii)]  We use Theorem \ref{th6:main}(ii) and Lemma \ref{lem82:two orbits}(ii).

    \item[(iv)] By hypothesis and Theorem \ref{th23:orbitellmu}(iii),  the sizes of the orbit $\Ob_\rho^{(3)}$  and the orbits $\Os_\mu$ are distinct.

    \item[(v)]
By hypothesis, there exist $\gamma, \delta \in \F_q^*$ such that $\gamma^4 = -1/3, \delta^3 = -2\rho$. Moreover $q\equiv1\pmod{12}$ implies $q\equiv1\pmod4$, so, by  \cite[Section 1.5(ix)]{Hirs_PGFF}, there exists $\iota \in \F_q^*$ such that $\iota^2 = -1$.
 We construct a projectivity $\Psi\in G_q$ such that $\ell_{-1/3}\Psi =\Ls_{\rho}$. The general form of the matrix $\MM^\Psi$ corresponding to $\Psi$ is given by \eqref{eq21:M}. We have:
\begin{equation*}
[1,0,1,0]\times\MM^\Psi=[a^3 + 3ab^2, a^2c + 2abd + b^2c,   ac^2 + ad^2 + 2bcd,   c^3 + 3cd^2]=[\rho,0,\eta,1].
\end{equation*}
 \begin{equation*}
[0,-1/3,0,1]\times\MM^\Psi=
\end{equation*}
 \begin{equation*}
[-a^2b + b^3,   -1/3a^2d - 2/3abc + b^2d,   -2/3acd - 1/3bc^2 + bd^2 ,  -c^2d+ d^3]=[\rho,0,\eta',1].
\end{equation*}
This implies $a,b,c,d \neq 0$, so we put $a=1$ as $\MM^\Psi$ is defined up to a factor of proportionality. Moreover
\begin{equation}\label{th_equiv_ell_L_1}
c + 2bd + b^2c=0.
\end{equation}
 \begin{equation}\label{th_equiv_ell_L_2}
 -d - 2bc + 3b^2d=0.
\end{equation}
 \begin{equation}\label{th_equiv_ell_L_3}
(1 + 3b^2)/( c^3 + 3cd^2)=
(-b + b^3)/ (-c^2d+ d^3)=\rho.
\end{equation}
By \eqref{th_equiv_ell_L_1}, $c = -2bd/(b^2+1)$. Substituting the value of $c$ in \eqref{th_equiv_ell_L_2} we obtain
 \begin{equation*}
d(3b^4 + 6*b^2 - 1)=3d(b + 1/2(3\iota + 3)\gamma^3 + 1/2(-\iota + 1)\gamma)
(b + 1/2(-3\iota + 3)\gamma^3 + 1/2(\iota + 1)\gamma)
\end{equation*}
 \begin{equation*}
(b + 1/2(3\iota - 3)\gamma^3 + 1/2(-\iota - 1)\gamma)
(b + 1/2(-3\iota - 3)\gamma^3 + 1/2(\iota - 1)\gamma)=0.
\end{equation*}
Therefore we we can put $b= - 1/2(3\iota + 3)\gamma^3 - 1/2(-\iota + 1)\gamma$. Then \eqref{th_equiv_ell_L_3} becomes $(\iota + 1)\gamma^3/d^3=\rho$, whence
$d^3=(\iota + 1)\gamma^3/\rho=-2(\iota + 1)\gamma^3/(-2\rho)=(1-\iota )^3\gamma^3/\delta^3$, so we can put $d=(1-\iota )\gamma/\delta$. Finally, $ad-bc=d-bc\neq 0$. In fact $d-bc=((3\iota + 3)\gamma^3 + (-3\iota + 3)\gamma)/\delta$. If $(\iota + 1)\gamma^2 + -\iota + 1=0$, then  $\gamma^2 = (\iota - 1)/(\iota + 1)=(\iota - 1)^2/((\iota + 1)(\iota - 1))=\iota$, whence $-1/3 = \gamma^4= \iota^2=-1$, contradiction as $q$ is odd.  \qedhere
  \end{description}
\end{proof}

\begin{theorem}\label{th82:q=-1mod3}
Let $q\equiv-1\pmod3$ be odd. Then all $\Ls_\rho$-lines generate the same $\frac{1}{2}(q^3-q)$-orbit $\Ob_1$ that is the orbit $\Os_\Lc$ \emph{\cite[Lemma 3.4(i), Theorem 3.5(iv)]{DMP_OrbLineO6MJOM}}. Moreover, this orbit $\Ob_1$ is different from any orbit $\Os_\mu$ of \emph{\cite[Section 7]{DMP_OrbLineO6MJOM}} except when $q\equiv  -1 \pmod{12}$; in this case the orbit  $\Ob_1$ coincides with the orbit $\Os_{-1/3}$ generated by the line $\ell_{-1/3}$ of \emph{\cite{DMP_OrbLineO6MJOM}}.
\end{theorem}

\begin{proof}
If $q\equiv  -1 \pmod{12}$ then $q\equiv  -1 \pmod{3}$ and by \cite[Section 1.5(xii)]{Hirs_PGFF} $-1/3$ is not a square. Then we use Lemma \ref{lem82:UniqueOrbit_qm1} and\cite[Theorem 7.7]{DMP_OrbLineO6MJOM}.\\
\end{proof}

\textbf {Acknowledgments}
The research of S. Marcugini and F. Pambianco was supported in part by the Italian
National Group for Algebraic and Geometric Structures and their Applications (GNSAGA -
INDAM) (Contract No. U-UFMBAZ-2019-000160, 11.02.2019) and by University of Perugia (Project No. 98751: Strutture Geometriche, Combinatoria e loro Applicazioni, Base Research Fund 2017--2019; Fighting Cybercrime with OSINT, Research Fund 2021). \\

\textbf {Data availability} Not applicable.\\

\textbf  {Conflict of interest} On behalf of all authors, the corresponding author states
that there is no conflict of interest. \\


\begin{thebibliography}{99}
\bibitem{BDMP-TwCub} Bartoli, D., Davydov, A.A., Marcugini, S., Pambianco, F.:
On planes through points off the twisted cubic in PG(3,q) and multiple covering codes, Finite Fields Appl. \textbf{67}, Article 101710 (2020). 
\url{https://doi.org/10.1016/j.ffa.2020.101710}.


\bibitem{BallicoCos} Ballico, E., Cossidente, A.: Curves of the projective 3-space, tangent
developables and partial spreads, Bull. Belg. Math. Soc. \textbf{7}, 387--394 (2000). 
\url{https://doi.org/10.36045/bbms/1103055653}.

\bibitem{BlokPelSzo} Blokhuis, A., Pellikaan, R., Sz\"{o}nyi, T.: The extended coset leader weight enumerator
of a twisted cubic code, Des. Codes Cryptogr. \textbf{90}(9), 2223--2247 (2022). 
\url{https://doi.org/10.1007/s10623-022-01060-0}.

\bibitem{BonPolvTwCub} Bonoli, G., Polverino, O.: The twisted cubic in $\PG(3, q)$ and translation spreads in $H(q)$, Discrete Math. \textbf{296}, 129--142 (2005). 
 \url{https://doi.org/10.1016/j.disc.2005.03.010}.

\bibitem{BrHirsTwCub}  Bruen, A.A., Hirschfeld, J.W.P.: Applications of line geometry over finite fields I: The twisted cubic, Geom. Dedicata \textbf{6}, 495--509  (1977). 
\url{https://doi.org/10.1007/BF00147786}.

\bibitem{CKS} Caputo, S., Korchm\'aros, G., Sonnino, A.: Multilevel secret sharing schemes arising from the normal rational curve, Discrete
Appl. Math. \textbf{284}, 158--165 (2020). 
\url{https://doi.org/10.1016/j.dam.2020.03.030}.

\bibitem{CLPolvT_Spr} Cardinali, I., Lunardon, G., Polverino, O., Trombetti, R.: Spreads in $H(q)$ and
1-systems of $Q(6,q)$, European J. Combin. \textbf{23}, 367--376 (2002). 
 \url{https://doi.org/10.1006/eujc.2001.0578}.

\bibitem{Carlitz} Carlitz L.: Explicit evaluation of certain exponential sums, Math. Scand. \textbf{44}(1), 5--16 (1979). 
\url{https://doi.org/10.7146/math.scand.a-11793}.

\bibitem{CerPavDM} Ceria, M., Pavese, F.: On the geometry of a $(q + 1)$-arc of $\PG(3, q)$, $q$ even, Discrete Math. \textbf{346}, Article 113594 (2023). 
\url{https://doi.org/10.1016/j.disc.2023.113594}.

\bibitem{CosHirsStTwCub}  Cossidente, A., Hirschfeld, J.W.P., Storme, L.: Applications of line geometry, III:
The quadric Veronesean and the chords of a twisted cubic, Austral. J. Combin. \textbf{16}, 99--111 (1997). 
\url{https://ajc.maths.uq.edu.au/pdf/16/ocr-ajc-v16-p99.pdf}.

\bibitem{DMP_RSCoset} Davydov, A.A., Marcugini, S., Pambianco, F.: On cosets weight distributions of the doubly-extended Reed-Solomon codes of codimension 4,
IEEE Trans. Inform. Theory \textbf{67}(8), 5088--5096 (2021). 
\url{https://doi.org/10.1109/TIT.2021.3089129}.

\bibitem{DMP_PointLineInc} Davydov, A.A., Marcugini, S., Pambianco, F.: Twisted cubic and point-line incidence matrix in $\PG(3,q)$, Des. Codes Cryptogr.  \textbf{89}(10), 2211--2233 (2021). 
\url{https://doi.org/10.1007/s10623-021-00911-6}.

\bibitem{DMP_PlLineIncJG} Davydov, A.A., Marcugini, S., Pambianco, F.: Twisted cubic and plane-line incidence matrix in $\PG(3,q)$, 	
J. Geom. \textbf{113}(2), Article 29 (2022). 
\url{https://doi.org/10.1007/s00022-022-00644-4}.

\bibitem{DMP_OrbLineMJOM} Davydov, A.A., Marcugini, S., Pambianco, F.: Orbits of lines for a twisted cubic
in $\PG(3,q)$, Mediterr. J. Math. \textbf{20}(3), Article 132 (2023). 
\url{https://doi.org/10.1007/s00009-023-02279-4}.

\bibitem{DMP_OrbLineO6MJOM} Davydov, A.A., Marcugini, S. Pambianco, F.: Orbits of the class $\OO_6$
 of lines external to the twisted cubic in $\PG(3,q)$, Mediterr. J. Math. \textbf{20}(3), Article 160 (2023). 
\url{https://doi.org/10.1007/s00009-023-02349-7}.

\bibitem{DMP_IncO6JG} Davydov, A.A., Marcugini, S., Pambianco, F.: Incidence matrices for the class $\OO_6$ of lines
external to the twisted cubic in $\PG(3,q)$, J. Geom. \textbf{114}(2), Article 21 (2023). 
\url{https://doi.org/10.1007/s00022-023-00678-2}.

\bibitem{DMP_arXivLrho2023}
Davydov, A.A., Marcugini, S., Pambianco, F.: Further results on orbits and
incidence matrices for the class $\mathcal{O}_6$ of lines
external to the twisted cubic in $\mathrm{PG}(3,q)$, arXiv:2401.00333v1 [math.CO] (2023). \url{https://doi.org/10.48550/arXiv.2401.00333}

\bibitem{WCC2024}
Davydov, A.A., Marcugini, S., Pambianco, F.: Further results on orbits and incidence matrices for the class $\mathcal{O}_6$ of lines external to the
 twisted cubic in $\mathrm{PG}(3,q)$. In: Proceedings XIII International Workshop on Coding and Cryptography, (WCC 2024), Perugia, Italy, June 2024, pp. 170-180 (2024).
 \url{https://wcc2024.sites.dmi.unipg.it/WCC_proceedings.pdf}

\bibitem{FanWangXu} Fan, C., Wang, A., Xu, L.: New classes of NMDS codes with dimension 3, Des. Codes Cryptogr. (2023). 
\url{https://doi.org/10.1007/s10623-023-01313-6}.

\bibitem{GiulVincTwCub} Giulietti, M.,  Vincenti, R.: Three-level secret sharing schemes from the twisted cubic,
Discrete Math. \textbf{310}, 3236--3240 (2010). 
\url{https://doi.org/10.1016/j.disc.2009.11.040}.

 \bibitem{GulLav} G\"{u}nay, G., Lavrauw, M.: On pencils of cubics on the projective line over finite fields of characteristic $> 3$,
 Finite Fields Appl. \textbf{78}, Article 101960 (2022). 
\url{https://doi.org/10.1016/j.ffa.2021.101960}.

\bibitem{Hirs_PG3q} Hirschfeld, J.W.P.: Finite Projective Spaces of Three Dimensions, Oxford Univ. Press, Oxford (1985)

\bibitem{Hirs_PGFF} Hirschfeld, J.W.P.: Projective Geometries over Finite Fields, 2nd edition, Oxford Univ.
Press, Oxford (1999)

\bibitem{KaPatPradOrbits} Kaipa, K., Patanker, N., Pradhan, P.:  On the $PGL_2(q)$-orbits of lines of $\PG(3, q)$ and binary
quartic forms, arXiv:2312.07118v2 [math.CO] (2024). 
\url{https://doi.org/10.48550/arXiv.2312.07118}.

\bibitem{KorchLanzSon}
Korchm\'{a}ros, G., Lanzone, V., Sonnino, A.:
Projective $k$-arcs and 2-level secret-sharing schemes, Des. Codes Cryptogr.  \textbf{64}(1), 3--15 (2012). 
\url{https://doi.org/10.1007/s10623-011-9562-5}.

\bibitem{LawLiPav}  Lavrauw, M., Lia, S., Pavese, F.: On the geometry of the Hermitian Veronese curve and its quasi-Hermitian surfaces,
Discrete Math. \textbf{346}, Article 113582 (2023). 
\url{https://doi.org/10.1016/j.disc.2023.113582}.

\bibitem{LunarPolv} Lunardon, G., Polverino, O.: On the twisted cubic of $\PG(3,q)$, J. Algebr. Combin.
\textbf{18}, 255--262 (2003). 
\url{https://doi.org/10.1023/B:JACO.0000011940.77655.b4}.

\bibitem{Maple}Maplesoft, a division of Waterloo Maple Inc.: Maple [Internet]. Waterloo, Ontario (2019). 
\url{https://www.maplesoft.com/products/maple/}.

\bibitem{groupbook}Thomas, A.D., Wood, G.V.: Group Tables. Shiva mathematics series, vol. 2. Shiva Publishing, Orpington (1980)

\bibitem{WolfNumbSolut} Wolfmann, J.: The number of solutions of certain diagonal equations over finite fields, J. Number Theory \textbf{42}(3) 247--257  (1992). 
\url{https://doi.org/10.1016/0022-314X(92)90091-3}.

\bibitem{ZanZuan2010} Zannetti, M.,  Zuanni, F.: Note on three-character $(q + 1)$-sets in $\PG(3, q)$,  Austral. J. Combin. \textbf{47}, 37--40 (2010). 
\url{https://ajc.maths.uq.edu.au/pdf/47/ajc_v47_p037.pdf}.


\bibitem{ZinHellesCharp} Zinoviev, V.A., Helleseth T., Charpin P.: On cosets of weight 4 of binary BCH codes with minimum distance 8 and exponential sums,
Probl. Inform. Transmission, \textbf{41}(4), 331--348 (2005). 
\url{https://doi.org/10.1007/s11122-006-0003-4}.
\end{thebibliography}
\end{document}